\documentclass[reqno,11pt]{amsart}

\usepackage[T1]{fontenc}
\usepackage{amsfonts,amssymb,amsmath,amsthm}
\usepackage{mathtools}
\usepackage[mathscr]{euscript}
\usepackage{fullpage}
\usepackage{enumitem}
\usepackage{cite}
\usepackage{xcolor}
\usepackage{tikz}
\usepackage{hyperref}

\usetikzlibrary{
  graphs,
  patterns,
  decorations.markings,
  decorations.pathmorphing,
  decorations.pathreplacing,
  arrows,
  matrix,
  calc,
  shapes
}

\numberwithin{equation}{section}
\allowdisplaybreaks

\newtheorem{thm}{Theorem}[section]
\newtheorem{prop}[thm]{Proposition}
\newtheorem{lem}[thm]{Lemma}

\theoremstyle{definition}
\newtheorem{defn}[thm]{Definition}
\newtheorem{ex}[thm]{Example}
\theoremstyle{remark}
\newtheorem{rmk}[thm]{Remark}

\renewcommand{\leq}{\leqslant}
\renewcommand{\geq}{\geqslant}

\newcommand{\Des}{\text{Des}}
\newcommand{\Inv}{\text{Inv}}
\newcommand{\des}{\text{des}}
\newcommand{\maj}{\text{maj}}
\newcommand{\inv}{\text{inv}}

\newcommand{\mstc}{\text{mstc}}
\newcommand{\code}{\text{code}}
\newcommand{\st}{\text{st}}
\newcommand{\std}{\text{std}}
\newcommand{\qbinom}[1]{\qbinomaux#1\relax}
\def\qbinomaux#1,#2\relax{\genfrac{(}{)}{0pt}{}{#1}{#2}_{q}}

\newcommand{\onebar}{\mathord{\vert}}
\newcommand{\twobars}{\mathord{\vert\mkern2mu\vert}}
\newcommand{\barwt}[2]{\underset{\scriptstyle #2}{\mathord{#1}}}
\newcommand{\nogap}{\phantom{\onebar}}

\tikzset{
  doubleword/.style={
    matrix of math nodes,
    row sep=5pt,
    column sep=2pt,
    nodes={inner sep=0pt,minimum width=1.2em,minimum height=1.35em},
    column 1/.style={nodes={minimum width=4.8em,anchor=east}}
  },
  exampleword/.style={
    matrix of math nodes,
    row sep=5pt,
    column sep=2pt,
    nodes={inner sep=0pt,minimum width=1.2em,minimum height=1.35em}
  },
  compartment bar/.style={line width=.55pt},
  beta bar/.style={compartment bar},
  double word entry/.style={
    inner sep=0pt,
    minimum width=1.5em,
    minimum height=1.35em
  },
  double word row label/.style={anchor=east,inner sep=0pt},
  double word descent weight/.style={anchor=north,inner sep=1pt},
  double word letter weight/.style={inner sep=0pt}
}

\newcommand{\barreddoubleword}[6][]{%
\begin{tikzpicture}[baseline=(current bounding box.center),#1]
  \def\dwN{#2}%
  \def\dwxstep{0.78}%
  \def\dwrowy{-0.58}%
  \def\dwbarsep{0.09}%
  \node[double word row label] at (-0.78,0) {$w$};
  \node[double word row label] at (-0.78,\dwrowy) {$\widehat w$};
  \foreach \dwletter [count=\dwcolumn from 1] in {#3}{
    \node[double word entry]
      at ({(\dwcolumn-1)*\dwxstep},0) {$\dwletter$};
  }
  \foreach \dwletter [count=\dwcolumn from 1] in {#4}{
    \node[double word entry]
      at ({(\dwcolumn-1)*\dwxstep},\dwrowy) {$\dwletter$};
  }
  \path ({-0.5*\dwxstep},0) -- ({(\dwN-0.5)*\dwxstep},0);
  \if\relax\detokenize{#5}\relax\else
    \foreach \dwposition/\dwmultiplicity in {#5}{
      \foreach \dwcopy in {1,...,\dwmultiplicity}{
        \pgfmathtruncatemacro{\dwlevel}{\dwcopy/2}
        \ifnum\dwcopy=1
          \def\dwoffset{0}%
        \else
          \ifodd\dwcopy
            \pgfmathsetmacro{\dwoffset}{-\dwlevel*\dwbarsep}%
          \else
            \pgfmathsetmacro{\dwoffset}{\dwlevel*\dwbarsep}%
          \fi
        \fi
        \pgfmathsetmacro{\dwbarx}
          {(\dwposition-0.5)*\dwxstep+\dwoffset}
        \draw[compartment bar] (\dwbarx,0.30) -- (\dwbarx,-0.84);
      }
    }
  \fi
  \if\relax\detokenize{#6}\relax\else
    \foreach \dwposition in {#6}{
      \pgfmathsetmacro{\dwbarx}{(\dwposition-0.5)*\dwxstep}
      \node[double word descent weight] at (\dwbarx,-0.90)
        {$\scriptstyle xq^{\dwposition}$};
    }
  \fi
\end{tikzpicture}%
}

\newcommand{\insertiondoubleword}[7][]{%
\begin{tikzpicture}[baseline=(current bounding box.center),#1]
  \def\dwN{#2}%
  \def\dwxstep{0.78}%
  \def\dwrowy{-0.58}%
  \def\dwweighty{-1.16}%
  \def\dwbarsep{0.09}%
  \node[double word row label] at (-0.78,0) {$w$};
  \node[double word row label] at (-0.78,\dwrowy) {$\widehat w$};
  \foreach \dwletter [count=\dwcolumn from 1] in {#3}{
    \node[double word entry]
      at ({(\dwcolumn-1)*\dwxstep},0) {$\dwletter$};
  }
  \foreach \dwletter [count=\dwcolumn from 1] in {#4}{
    \node[double word entry]
      at ({(\dwcolumn-1)*\dwxstep},\dwrowy) {$\dwletter$};
  }
  \path ({-0.5*\dwxstep},\dwweighty)
    -- ({(\dwN-0.5)*\dwxstep},\dwweighty);
  \if\relax\detokenize{#5}\relax\else
    \foreach \dwposition/\dwmultiplicity in {#5}{
      \foreach \dwcopy in {1,...,\dwmultiplicity}{
        \pgfmathtruncatemacro{\dwlevel}{\dwcopy/2}
        \ifnum\dwcopy=1
          \def\dwoffset{0}%
        \else
          \ifodd\dwcopy
            \pgfmathsetmacro{\dwoffset}{-\dwlevel*\dwbarsep}%
          \else
            \pgfmathsetmacro{\dwoffset}{\dwlevel*\dwbarsep}%
          \fi
        \fi
        \pgfmathsetmacro{\dwbarx}
          {(\dwposition-0.5)*\dwxstep+\dwoffset}
        \draw[compartment bar] (\dwbarx,0.30) -- (\dwbarx,-0.84);
      }
    }
  \fi
  \if\relax\detokenize{#7}\relax\else
    \foreach \dwposition in {#7}{
      \pgfmathsetmacro{\dwletterx}{(\dwposition-1)*\dwxstep}
      \draw[beta bar] (\dwletterx,0.30) -- (\dwletterx,-0.84);
    }
  \fi
  \if\relax\detokenize{#6}\relax\else
    \foreach \dwposition/\dwweight in {#6}{
      \pgfmathsetmacro{\dwletterx}{(\dwposition-1)*\dwxstep}
      \node[double word letter weight] at (\dwletterx,\dwweighty)
        {$\scriptstyle \dwweight$};
    }
  \fi
\end{tikzpicture}%
}

\newcommand{\weighteddoubleword}[8][]{%
\begin{tikzpicture}[baseline=(current bounding box.center),#1]
  \def\dwN{#2}%
  \def\dwxstep{0.78}%
  \def\dwrowy{-0.58}%
  \def\dwbarweighty{-0.90}%
  \def\dwletterweighty{0.62}%
  \def\dwbarsep{0.09}%
  \node[double word row label] at (-0.78,0) {$w$};
  \node[double word row label] at (-0.78,\dwrowy) {$\widehat w$};
  \foreach \dwletter [count=\dwcolumn from 1] in {#3}{
    \node[double word entry]
      at ({(\dwcolumn-1)*\dwxstep},0) {$\dwletter$};
  }
  \foreach \dwletter [count=\dwcolumn from 1] in {#4}{
    \node[double word entry]
      at ({(\dwcolumn-1)*\dwxstep},\dwrowy) {$\dwletter$};
  }
  \path ({-0.5*\dwxstep},\dwletterweighty)
    -- ({(\dwN-0.5)*\dwxstep},\dwletterweighty);
  \if\relax\detokenize{#5}\relax\else
    \foreach \dwposition/\dwmultiplicity in {#5}{
      \foreach \dwcopy in {1,...,\dwmultiplicity}{
        \pgfmathtruncatemacro{\dwlevel}{\dwcopy/2}
        \ifnum\dwcopy=1
          \def\dwoffset{0}%
        \else
          \ifodd\dwcopy
            \pgfmathsetmacro{\dwoffset}{-\dwlevel*\dwbarsep}%
          \else
            \pgfmathsetmacro{\dwoffset}{\dwlevel*\dwbarsep}%
          \fi
        \fi
        \pgfmathsetmacro{\dwbarx}
          {(\dwposition-0.5)*\dwxstep+\dwoffset}
        \draw[compartment bar] (\dwbarx,0.30) -- (\dwbarx,-0.84);
      }
    }
  \fi
  \if\relax\detokenize{#8}\relax\else
    \foreach \dwposition in {#8}{
      \pgfmathsetmacro{\dwletterx}{(\dwposition-1)*\dwxstep}
      \draw[beta bar] (\dwletterx,0.30) -- (\dwletterx,-0.84);
    }
  \fi
  \if\relax\detokenize{#6}\relax\else
    \foreach \dwposition in {#6}{
      \pgfmathsetmacro{\dwbarx}{(\dwposition-0.5)*\dwxstep}
      \node[double word descent weight] at (\dwbarx,\dwbarweighty)
        {$\scriptstyle xq^{\dwposition}$};
    }
  \fi
  \if\relax\detokenize{#7}\relax\else
    \foreach \dwposition/\dwweight in {#7}{
      \pgfmathsetmacro{\dwletterx}{(\dwposition-1)*\dwxstep}
      \node[double word letter weight] at (\dwletterx,\dwletterweighty)
        {$\scriptstyle \dwweight$};
    }
  \fi
\end{tikzpicture}%
}

\newcommand{\manualweighteddoubleword}[8][]{%
\begin{tikzpicture}[baseline=(current bounding box.center),#1]
  \def\dwN{#2}%
  \def\dwxstep{0.78}%
  \def\dwrowy{-0.58}%
  \def\dwbarweighty{-0.90}%
  \def\dwletterweighty{0.62}%
  \def\dwbarsep{0.09}%
  \node[double word row label] at (-0.78,0) {$w$};
  \node[double word row label] at (-0.78,\dwrowy) {$\widehat w$};
  \foreach \dwletter [count=\dwcolumn from 1] in {#3}{
    \node[double word entry]
      at ({(\dwcolumn-1)*\dwxstep},0) {$\dwletter$};
  }
  \foreach \dwletter [count=\dwcolumn from 1] in {#4}{
    \node[double word entry]
      at ({(\dwcolumn-1)*\dwxstep},\dwrowy) {$\dwletter$};
  }
  \path ({-0.5*\dwxstep},\dwletterweighty)
    -- ({(\dwN-0.5)*\dwxstep},\dwletterweighty);
  \if\relax\detokenize{#5}\relax\else
    \foreach \dwposition/\dwmultiplicity in {#5}{
      \foreach \dwcopy in {1,...,\dwmultiplicity}{
        \pgfmathtruncatemacro{\dwlevel}{\dwcopy/2}
        \ifnum\dwcopy=1
          \def\dwoffset{0}%
        \else
          \ifodd\dwcopy
            \pgfmathsetmacro{\dwoffset}{-\dwlevel*\dwbarsep}%
          \else
            \pgfmathsetmacro{\dwoffset}{\dwlevel*\dwbarsep}%
          \fi
        \fi
        \pgfmathsetmacro{\dwbarx}
          {(\dwposition-0.5)*\dwxstep+\dwoffset}
        \draw[compartment bar] (\dwbarx,0.30) -- (\dwbarx,-0.84);
      }
    }
  \fi
  \if\relax\detokenize{#8}\relax\else
    \foreach \dwposition in {#8}{
      \pgfmathsetmacro{\dwletterx}{(\dwposition-1)*\dwxstep}
      \draw[beta bar] (\dwletterx,0.30) -- (\dwletterx,-0.84);
    }
  \fi
  \if\relax\detokenize{#6}\relax\else
    \foreach \dwposition/\dwqweight in {#6}{
      \pgfmathsetmacro{\dwbarx}{(\dwposition-0.5)*\dwxstep}
      \node[double word descent weight] at (\dwbarx,\dwbarweighty)
        {$\scriptstyle x\dwqweight$};
    }
  \fi
  \if\relax\detokenize{#7}\relax\else
    \foreach \dwposition/\dwweight in {#7}{
      \pgfmathsetmacro{\dwletterx}{(\dwposition-1)*\dwxstep}
      \node[double word letter weight] at (\dwletterx,\dwletterweighty)
        {$\scriptstyle \dwweight$};
    }
  \fi
\end{tikzpicture}%
}

\author{Alexandr Garbali}

\address{Alexandr Garbali, School of Mathematics and Statistics, The University of Melbourne, Australia}
\email{alexandr.garbali@unimelb.edu.au}

\title[Hypergeometric positivity]{Positivity of the basic hypergeometric series and the MacMahon identity}

\begin{document}

\begin{abstract}
The basic hypergeometric series ${}_{n+1}\phi_{n}$ whose upper and lower parameters are non-negative integer powers of $q$ satisfying a certain constraint can be, after a normalisation, computed as a positive polynomial \cite{GW}. We derive three formulas for this polynomial. The first formula is a generalisation of the MacMahon identity in which an infinite binomial series is identified with a normalised generating function of $\des$ and $\maj$ statistics of multiset permutations. The second formula is a generating function of multiset permutations for the inversion statistic and a generalisation of the statistic $\mstc$ of \cite{Carnevale}. The third formula is a finite sum containing $2n$ $q$-binomial coefficients.
\end{abstract}

\maketitle




\section{Introduction}
Let $\alpha=(\alpha_0,\alpha_1\ldots \alpha_n)$ and $\beta=(\beta_1\ldots \beta_n)$ be two collections of non-negative integers. Consider the following basic hypergeometric series 
\begin{align}\label{eq: hgf intro}
\phi_{\alpha,\beta}(x;q):= 
{}_{n+1}\phi_n\!\left[\begin{matrix}q^{\alpha_0+1},q^{\alpha_1+1},\ldots,q^{\alpha_n+1}\\ q^{\beta_1+1},\ldots,q^{\beta_n+1}\end{matrix};q,x\right]:=
\sum_{k\geq 0}x^k\frac{(q^{\alpha_0+1};q)_k\prod_{i=1}^{n}(q^{\alpha_i+1};q)_k}{(q;q)_k\prod_{i=1}^{n}(q^{\beta_i+1};q)_k}
\end{align}
where $\phi_{\alpha,\beta}(x;q)$ will be used as a shorthand notation and $(z;q)_k$ is the $q$-Pochhammer symbol
\begin{align}\label{eq: q-Pochhammer}
(z;q)_k = \prod_{i=0}^{k-1}(1-z q^i)
\end{align}
In \cite{GW} it was proved that the series \eqref{eq: hgf intro} can be summed and, when normalised, produces a positive polynomial in $x$ and $q$ provided the following conditions on $\alpha$ and $\beta$ are satisfied\footnote{For any permutations $\sigma\in S_{n+1}$ and $\pi\in S_{n}$ let $\alpha'=\sigma(\alpha)$ and $\beta'=\pi(\beta)$, then $\phi_{\alpha,\beta}(x;q)=\phi_{\alpha',\beta'}(x;q)$. One may use any such pair $\alpha',\beta'$ instead of $\alpha,\beta$ as long as the corresponding conditions \eqref{eq: positive regime intro} hold.}
\begin{align}\label{eq: positive regime intro}
\alpha_{i-1}\geq \beta_i,
 \quad
 \text{and}
 \quad 
\alpha_{i}\geq \beta_i,
\qquad
\text{for }
i=1\ldots n
\end{align}
This was done in \cite{GW} by explicitly resumming \eqref{eq: hgf intro} and writing it as a finite sum of products of $\sim n^2$ $q$-binomial coefficients
\begin{align}\label{eq:qbinom}
\qbinom{a,b}=
\begin{cases}
\frac{(q;q)_a}{(q;q)_{a-b}(q;q)_b} \quad  &  a\geq b\geq 0\\
0 &  \text{otherwise}
\end{cases}
\end{align}
In this paper we derive three formulas for the series \eqref{eq: hgf intro}: two formulas in terms of generating functions of certain statistics on multiset permutations and one formula in terms of a finite sum of products of $2n$ $q$-binomial coefficients. 

Let $\alpha_{i}\geq \beta_i$ for $i=1\ldots n$ and define two compositions $\delta=(\delta_0\ldots \delta_n)$ and $\gamma=(\gamma_1\ldots \gamma_{n+1})$ by
\begin{align}
\label{eq:delta gamma}
\delta_0=\gamma_{n+1}=\alpha_0,\qquad 
\delta_i=\gamma_i=\alpha_i-\beta_i, \hspace{0.5cm} i=1\ldots n
\end{align} 
and set $N=\delta_0+\cdots + \delta_n=\gamma_1+\cdots +\gamma_{n+1}$. We will assume throughout that
\begin{align}\label{eq:alpha_0 condition}
\alpha_0\geq\max(\beta)
\end{align}

Define $S_\delta$ to be the set of words given by the permutations of the multiset $(0^{\delta_0}\ldots n^{\delta_n})$ and $S_\gamma$ to be the set of words given by the permutations of the multiset $(1^{\gamma_1}\ldots (n+1)^{\gamma_{n+1}})$, then our three formulas read
\begin{align}\label{eq:phi des maj}
\phi_{\alpha,\beta}(x;q)
&=\frac{1}{\mathcal N(x,q)}
 \sum_{w\in S_{\delta}} x^{\des_\beta(w)} q^{\maj_\beta(w)+\inv_\beta(w)}\\
\label{eq:phi mstc inv}
\phi_{\alpha,\beta}(x;q)
&=\frac{1}{\mathcal N(x,q)}
 \sum_{w\in S_{\gamma}} x^{\mstc_\beta(w)} q^{\inv(w)}
\end{align}
\begin{align}\label{eq:phi binom}
\phi_{\alpha,\beta}(x;q)
&=\frac{1}{\mathcal N(x,q)}
\sum_{\kappa}
x^{\kappa_1}
\prod_{i=1}^n
q^{(\kappa_i-\kappa_{i+1})(\beta_i+\kappa_i)}
\qbinom{\alpha_i+\kappa_{i+1},\beta_i+\kappa_i}
\qbinom{\lambda_{i}-\kappa_{i+1},
\lambda_{i}-\kappa_{i}}
\end{align}
where the function $\mathcal N(x,q)$ is the normalising factor
\begin{align}
\label{eq: norm}
\mathcal N(x,q):= (x;q)_{N+1}\prod_{i=1}^n \qbinom{\alpha_i,\beta_i}
\end{align}
In the first formula \eqref{eq:phi des maj} the three statistics $\des_\beta,\maj_\beta$ and $\inv_\beta$ are modifications of the classical statistics $\des,\maj$ and $\inv$ (see Definitions \ref{def:classical statistics} and \ref{def:beta statistics}). For this formula to work we require the ordering $\beta_1\geq \ldots \geq \beta_n$. 
In the second formula \eqref{eq:phi mstc inv} $\mstc_\beta$ is a statistic generalising the known statistic $\mstc$ of \cite{Carnevale} (see Definitions \ref{def:mstc} and \ref{def:mstc beta}). In the third formula \eqref{eq:phi binom} the sum is over integer compositions $\kappa=(\kappa_1\ldots \kappa_n)$ and we set $\kappa_{n+1}=0$. The integers $\lambda_i$ are defined by 
\begin{align*}
\lambda_i=\gamma_{i+1}+\cdots+\gamma_{n+1}-\beta_i, \qquad i=1\ldots n 
\end{align*}

In particular, the identity \eqref{eq:phi binom} can be used to derive a branching formula for the function $\phi_{\alpha,\beta}(x;q)$. Set 
$\alpha^{(r)}=(\alpha_0+\alpha_n-\beta_n, \alpha_1+r\ldots\alpha_{n-1}+r)$ and $\beta^{(r)}=(\beta_1+r\ldots \beta_{n-1}+r)$ then
\begin{align}
\label{eq:hgf branching intro}
\phi_{\alpha,\beta}(x;q)=\sum_{r= 0}^{\max(\alpha_0,\alpha_n)-\beta_n}(xq^{\alpha_0^{(r)}+1})^r
\frac{(q^{\beta_n-\alpha_0};q)_r (q^{\beta_n-\alpha_n};q)_r}{(q;q)_r(q^{\beta_n+1};q)_r}\prod_{i=1}^{n-1}\frac{(q^{\alpha_i+1};q)_r}
     {(q^{\beta_i+1};q)_r}
\phi_{\alpha^{(r)},\beta^{(r)}}(x;q)
\end{align}
The three formulas \eqref{eq:phi des maj}, \eqref{eq:phi mstc inv}, \eqref{eq:phi binom} and the property \eqref{eq:hgf branching intro} imply that the function $\phi_{\alpha,\beta}(x;q)$ has a rich combinatorial structure and deserves further study.

We note that the hypergeometric function $\phi_{\alpha,\beta}(x;q)$ can be written in terms of the following normalised binomial series
\begin{align}
\label{eq:def Phi}
\Phi_{\alpha,\beta}(x;q) : = 
(x;q)_{N+1} \sum_{k\geq 0}x^k \qbinom{k+\alpha_0,k}\prod_{i=1}^{n}\qbinom{k+\alpha_i,k+\beta_i}
\end{align}
This function is related to the hypergeometric function $\phi_{\alpha,\beta}(x;q)$ by the formula
\begin{align}
\label{eq:hgf binomial}
\Phi_{\alpha,\beta}(x;q) =\mathcal N(x,q) \phi_{\alpha,\beta}(x;q)
\end{align}
In the rest of the paper we will work with the function $\Phi_{\alpha,\beta}(x;q)$ instead of $\phi_{\alpha,\beta}(x;q)$ since this function is a generalisation of the function 
\begin{align}
\label{eq:def Phi alpha}
\Phi_{\alpha}(x;q) : = 
(x;q)_{\alpha_0+\cdots + \alpha_n+1} \sum_{k\geq 0}x^k \prod_{i=0}^{n}\qbinom{k+\alpha_i,k}
\end{align}
which has been widely studied starting from MacMahon \cite{MacMahon}. Below we give more details regarding the three formulas \eqref{eq:phi des maj}, \eqref{eq:phi mstc inv} and \eqref{eq:phi binom}. 

\subsection{The generalised MacMahon identity}\label{sec:MacMahon}
The formula \eqref{eq:phi des maj} when expressed via $\Phi_{\alpha,\beta}(x;q)$ can be viewed as a generalisation of the MacMahon identity \cite{MacMahon}. 

Consider a composition of positive integers $\alpha=(\alpha_1\ldots \alpha_n)$ and let $S_\alpha$ be the set of all permutations of the multiset $\{1^{\alpha_1}\ldots n^{\alpha_n}\}$. Set $M=\alpha_1+\cdots + \alpha_n$. 

\begin{defn}\label{def:classical statistics}
For a word $w\in S_\alpha$ define
\begin{align}\label{eq:classical sets}
\Des(w)&=\{i:1\leq i<M,\ w_i>w_{i+1}\}\\
\Des^c(w)&=\{i:1\leq i<M,\ w_i\leq w_{i+1}\}\\
\Inv(w)&=\{(i,j):1\leq i<j\leq M,\ w_i>w_j\}
\end{align}
The corresponding statistics are
\begin{align}\label{eq:classical statistics}
\des(w)=|\Des(w)|,
\qquad
\maj(w)=\sum_{i\in\Des(w)}i,
\qquad
\inv(w)=|\Inv(w)|
\end{align}
\end{defn}

The MacMahon identity \cite{MacMahon} for multiset permutations reads
\begin{align}\label{eq:Mac identity}
(x;q)_{M+1} \sum_{k\geq 0}x^k \prod_{i=1}^{n}\qbinom{k+\alpha_i,k}
=
\sum_{w\in S_\alpha} x^{\des(w)} q^{\maj(w)}
\end{align}

One may ask if it is possible to add the bottom integer parameters $(\beta_1\ldots \beta_n)$ (with $\beta_i\leq \alpha_i$ and $K=\beta_1+\cdots+ \beta_n$) inside the binomial coefficients on the left hand side
\begin{align*}
(x;q)_{M-K+1}  \sum_{k\geq 0}x^k \prod_{i=1}^{n}\qbinom{k+\alpha_i,k+\beta_i}
\end{align*}
compute the infinite sum and again obtain a positive polynomial with powers of $x$ and $q$ having some meaning similar to \eqref{eq:Mac identity}. The smallest non-trivial example shows that the resulting polynomial has a negative coefficient
\begin{align*}
(1-x)(1-qx)\sum_{k\geq 0}x^k\qbinom{k+2,k+1}=\frac{(1-x)(1-qx)}{1-q}
 \left(\frac{1}{1-x}-\frac{q^2}{1-qx}\right)
 =1+q-qx
\end{align*}
The positivity can be restored by inserting another binomial with the pair $(\alpha_0,\beta_0)$ satisfying the conditions $\alpha_0\geq \max(\beta)$  and $\beta_0=0$
\begin{align*}
(x;q)_{M-K+\alpha_0+1} \sum_{k\geq 0}x^k \qbinom{k+\alpha_0,k}\prod_{i=1}^{n}\qbinom{k+\alpha_i,k+\beta_i}
\end{align*}
This infinite series can be summed and the result can be written as a sum over the set $S_\delta$ of the multiset  permutations of $\{0^{\delta_0}\ldots n^{\delta_n}\}$ where $\delta_i=\alpha_i-\beta_i$, $i=0\ldots n$. The new statistics $\des_\beta$ and $\maj_\beta$ (generalising $\des$ and $\maj$) now must depend on the parameters $\beta_1\ldots \beta_n$ and the letters $0$ in $w$ must play a distinguished role. 
\begin{defn}\label{def:beta statistics}
For $w\in S_\delta$ define a word $\widehat w$ by
\begin{align}\label{eq:hat word}
\widehat w_i=
\begin{cases}
|\{j<i:w_j=0\}|, & w_i=0\\
\beta_{w_i}, & w_i>0
\end{cases}
\end{align}
and set
\begin{align}\label{eq:D beta}
\Des_\beta(w)&=\Des(w)\cap\Des^c(\widehat w)\\
\Inv_\beta(w)&=\Inv(w)\cap\Inv(\widehat w)
\end{align}
The corresponding statistics are
\begin{align}\label{eq:beta statistics}
\des_\beta(w)=|\Des_\beta(w)|,
\qquad
\maj_\beta(w)=\sum_{i\in\Des_\beta(w)}i,
\qquad
\inv_\beta(w)=|\Inv_\beta(w)|
\end{align}
\end{defn}

For a given word $w$ the word $\widehat w$ can be read directly from it. For example
\[
\begin{array}{c@{\;}ccccccc}
w={}&2&0&1&2&0&3&0\\
\widehat w={}&\beta_2&0&\beta_1&\beta_2&1&\beta_3&2
\end{array}
\]
We present the formula \eqref{eq:phi des maj} as the following statement.
\begin{thm}\label{thm:beta Mac}
Let $\alpha=(\alpha_0,\alpha_1\ldots \alpha_n)$ and $\beta=(\beta_1\ldots \beta_n)$ be two collections of non-negative integers satisfying:
\begin{align}
\label{eq:condition 1}
&\beta_i\geq \beta_{i+1}, \qquad \text{for } i=1\ldots n-1\\
\label{eq:condition 2}
&\alpha_i\geq \beta_{i}, \qquad \text{for } i=1\ldots n,
\qquad \alpha_0\geq \max(\beta)
\end{align}
Set $\delta=(\delta_0\ldots \delta_n)$, where $\delta_0 =\alpha_0$ and $\delta_i=\alpha_i-\beta_i$, $(i=1\ldots n)$. Let $N=\delta_0+\cdots +\delta_n$ and $S_\delta$ be the corresponding set of multiset permutations, then we have
\begin{align}\label{eq:beta Mac identity}
(x;q)_{N+1} \sum_{k\geq 0}x^k \qbinom{k+\alpha_0,k}\prod_{i=1}^{n}\qbinom{k+\alpha_i,k+\beta_i}
=
\sum_{w\in S_{\delta}} x^{\normalfont{\des}_\beta(w)} q^{\normalfont{\maj}_\beta(w)+\inv_\beta(w)}
\end{align}
\end{thm}
Recall the definition \eqref{eq:def Phi} of $\Phi_{\alpha,\beta}(x;q)$ and \eqref{eq:def Phi alpha} of $\Phi_{\alpha}(x;q)$. We write the MacMahon identity \eqref{eq:Mac identity} for $\alpha=(\alpha_0\ldots \alpha_n)$ and its generalisation \eqref{eq:beta Mac identity} as
\begin{align}
\label{eq:beta Phi des maj}
\Phi_{\alpha,\beta}(x;q) &= \sum_{w\in S_{\delta}} x^{\normalfont{\des}_\beta(w)} q^{\normalfont{\maj}_\beta(w)+\inv_\beta(w)}\\
\label{eq:Phi des maj}
\Phi_{\alpha}(x;q) &= \sum_{w\in S_{\alpha}} x^{\normalfont{\des}(w)} q^{\normalfont{\maj}(w)}
\end{align}


\subsection{The $(\mstc_\beta,\inv)$-formula}\label{sec:mstc inv}
A pair of statistics with the same joint distribution as $(\des,\maj)$ is called Euler--Mahonian. Such pairs have been studied on permutations and words \cite{Foata,Clarkeetal,Skandera}. Carnevale introduced the statistic $\mstc$ on multiset permutations and proved that $(\mstc,\inv)$ is Euler--Mahonian \cite{Carnevale} thus providing another formula for $\Phi_{\alpha}(x;q)$. Our second formula \eqref{eq:phi mstc inv} is equivalent to the analogue of Carnevale's result for $\Phi_{\alpha,\beta}(x;q)$.

Let $\alpha=(\alpha_1\ldots \alpha_n)$ be a non-negative integer composition and $M=\alpha_1+\cdots + \alpha_n$. Let $S_\alpha$ be the set of multiset permutations as before and $S_{M}$ the standard symmetric group on $M$ letters. 
\begin{defn}\label{def:mstc}
For $w\in S_\alpha$ define standardisation $\std(w)\in S_{M}$ by assigning the integers $1\ldots M$ to the letters of $w$ in increasing order and to equal letters from left to right (see Example \ref{ex:mstc beta}). For $\sigma\in S_M$ define its Lehmer code $\code(\sigma)=(c_1\ldots c_{M})$ by
\begin{align}\label{eq:Lehmer code}
c_i=|\{j:i<j\leq M,\ \sigma_i>\sigma_j\}|
\end{align}
For an integer word $c=(c_1\ldots c_M)$ set
\begin{align}\label{eq:st}
\st(c)
=
\max\bigl\{r\in\{0\ldots M\}:\exists\ 1\leq i_1<\cdots<i_r\leq M,\ c_{i_j}>r-j\text{ for }j=1\ldots r\bigr\}
\end{align}
and define
\begin{align}\label{eq:mstc}
\mstc(w):=\st\bigl(\code((\std(w))^{-1})\bigr)
\end{align}
\end{defn}

Carnevale's identity gives a second expression for $\Phi_{\alpha}(x;q)$
\begin{align}\label{eq:Phi mstc inv}
\Phi_{\alpha}(x;q)
=
\sum_{w\in S_\alpha}x^{\mstc(w)}q^{\inv(w)}
\end{align}

Now we define a $\beta$-analogue of $\mstc$. Let $\alpha=(\alpha_0,\alpha_1\ldots \alpha_n)$ and $\beta=(\beta_1\ldots \beta_n)$ be two non-negative integer compositions with $\alpha_i\geq\beta_i$ for $i=1\ldots n$ and set also $\beta_{n+1}=0$. Define $\gamma=(\gamma_1\ldots \gamma_{n+1})$ by 
\begin{align}
\label{eq:gamma}
\gamma_{n+1}=\alpha_0,\qquad 
\gamma_i=\alpha_i-\beta_i, \hspace{0.5cm} i=1\ldots n
\end{align}
and $N=\gamma_1+\cdots + \gamma_{n+1}$. Let $S_\gamma$ be the set of multiset permutations of $\{1^{\gamma_1}\ldots (n+1)^{\gamma_{n+1}}\}$ and $S_{N}$ the standard symmetric group on $N$ letters. 
\begin{defn}\label{def:mstc beta}
We denote by $\beta^\gamma$ the following word
\begin{align}\label{eq:beta^gamma}
\beta^\gamma:=(\beta_1^{\gamma_1}\ldots \beta_n^{\gamma_n},0^{\gamma_{n+1}})
\end{align}
and define
\begin{align}\label{eq:mstc beta}
\mstc_\beta(w):=\st\left(\code((\std(w))^{-1})-\beta^\gamma\right)
\end{align}
\end{defn}

\begin{ex}\label{ex:mstc beta}
Let $\alpha=(1,4,3)$, $\beta=(1,1)$ then $\gamma=(3,2,1)$. For $w=312211$ we have
\begin{align*}
\std(w)=614523
\qquad
(\std(w))^{-1}=256341
\end{align*}
and
\begin{align*}
c=\code\bigl((\std(w))^{-1}\bigr)
=133 11 0
\end{align*}
Thus $\inv(w)=c_1+\cdots +c_6=9$ and $\st(c)=3$ since there is a subword $331$ which is element-wise strictly greater than $210$. We compute $\beta^\gamma$
\begin{align*}
\beta^\gamma=(\beta_1^{\gamma_1},\beta_2^{\gamma_2},\beta_3^{\gamma_3})=(111110)
\end{align*}
Then $c-\beta^\gamma$ gives us the sequence $(022000)$ for which $\st$ is equal to $2$. Therefore $\mstc_\beta(w)=2$.
\end{ex}
\begin{thm}\label{thm:mstc inv}
Let $\alpha=(\alpha_0,\alpha_1\ldots \alpha_n)$ and $\beta=(\beta_1\ldots \beta_n)$ be two collections of integers satisfying:
\begin{align}
\label{eq:condition alpha beta mstc}
\alpha_0\geq \max(\beta),
\qquad \alpha_i\geq \beta_{i}\geq 0,
\qquad 
\text{for } i=1\ldots n
\end{align}
Set $\gamma_i=\alpha_i-\beta_i$ for $i=1\ldots n$ and $\gamma_{n+1}=\alpha_0$, then we have
\begin{align}\label{eq:mstc beta inv}
\Phi_{\alpha,\beta}(x;q)
=
 \sum_{w\in S_\gamma} x^{\normalfont{\mstc}_\beta(w)} q^{\normalfont{\inv}(w)}
\end{align}
\end{thm}


\subsection{The finite binomial formula}
We present the third formula \eqref{eq:phi binom} in terms of the normalised series $\Phi_{\alpha,\beta}(x;q)$. This formula is given by a finite sum of products of $q$-binomial coefficients and thus is manifestly polynomial and positive.
\begin{thm}\label{thm:qbinomial}
Let $\alpha=(\alpha_0,\alpha_1\ldots \alpha_n)$ and $\beta=(\beta_1\ldots \beta_n)$ be two collections of integers satisfying:
\begin{align}
\label{eq:condition alpha beta}
\alpha_0\geq \max(\beta),
\qquad \alpha_i\geq \beta_{i}\geq 0,\qquad i=1\ldots n
\end{align}
Set $\gamma_i=\alpha_i-\beta_i$ for $i=1\ldots n$ and $\gamma_{n+1}=\alpha_0$, define further
\begin{align}\label{eq:lambda}
\lambda_i=\gamma_{i+1}+\cdots+\gamma_{n+1}-\beta_i,\qquad i=1\ldots n
\end{align}
then we have
\begin{align}\label{eq:qbinomial formula}
\Phi_{\alpha,\beta}(x;q)
=
\sum_{\kappa}
x^{\kappa_1}
\prod_{i=1}^n
q^{(\kappa_i-\kappa_{i+1})(\beta_i+\kappa_i)}
\qbinom{\alpha_i+\kappa_{i+1},\beta_i+\kappa_i}
\qbinom{\lambda_{i}-\kappa_{i+1},
\lambda_{i}-\kappa_{i}}
\end{align}
where the sum is over integer compositions $\kappa=(\kappa_1\ldots \kappa_n)$  and $\kappa_{n+1}=0$.
\end{thm}

With the convention \eqref{eq:qbinom} it is easy to see that a summand in \eqref{eq:qbinomial formula} is non-zero only if
\begin{align}\label{eq:kappa conditions}
0\leq \kappa_i-\kappa_{i+1}\leq\gamma_i,
\qquad
\kappa_i\leq \lambda_{i}
\qquad (i=1,\ldots,n)
\end{align}
Therefore the summation over $\kappa$ is finite.
We note that if we order the parameters $\beta$ s.t. $\beta_i\leq\beta_{i+1}$ then this implies an ordering on $\lambda=(\lambda_1\ldots \lambda_n)$, i.e. $\lambda_i\geq \lambda_{i+1}$. Thus $\lambda$ can be viewed as an integer partition. The conditions on $\kappa$ in \eqref{eq:kappa conditions} also impose $\kappa_i\geq \kappa_{i+1}$ and together with $\kappa_i\leq \lambda_{i}$ it means that one can compute \eqref{eq:qbinomial formula} by summing over all integer partitions $\kappa$ whose Young diagrams are contained inside the Young diagram of $\lambda$.


\subsection{Remark on $\Phi_{\alpha,\beta}(x;q)$ in integrable models and symmetric functions}
The problem of finding formulas for $\Phi_{\alpha,\beta}(x;q)$ was originally motivated by a problem of positivity of the $q$-hypergeometric function as stated in \cite{GW}. The authors in \cite{GW} considered a construction of the modified Macdonald polynomials using an integrable lattice model approach \cite{CdGW}. It is known that these symmetric polynomials positively expand in the basis of monomial symmetric functions \cite{Haiman,HHL,AMM1}. The construction of \cite{GW} (see also \cite{JL}) of the modified Macdonald polynomials produced a monomial formula with the coefficients expressed via $\Phi_{\alpha,\beta}(x;q)$ as defined in \eqref{eq:def Phi} with the restriction \eqref{eq: positive regime intro} on the integer arguments. This suggested that $\Phi_{\alpha,\beta}(x;q)$ are positive polynomials in $x$ and $q$ under the conditions \eqref{eq: positive regime intro}. The authors \cite{GW} resummed the series \eqref{eq:def Phi} and obtained a finite $q$-binomial formula which proved the positivity of $\Phi_{\alpha,\beta}(x;q)$. The resulting formula of \cite{GW} however is much more complicated than \eqref{eq:qbinomial formula}. 

Let us briefly explain what $\Phi_{\alpha,\beta}(x;q)$ computes in the context of integrable lattice models. A key object is the so called $L$-matrix. The specific $L$-matrix that was used in \cite{GW} is an object associated to the quantum affine algebra $U_{q}(\widehat{\frak{sl}}_{N})$ but here we will need to discuss the $U_{q}(\widehat{\frak{sl}}_{2})$ case. This matrix is infinite dimensional  with the matrix elements denoted by $L_{\tilde k,k}^{s,r}$ (denoted $\mathcal L_{\rho,\tilde \rho}^{\sigma,\tilde \sigma}$ in \cite{GW}) where $\tilde k,k,s$ and $r$ are non-negative integers. The matrix elements $L_{\tilde k,k}^{s,r}$ vanish if $k+r\neq \tilde k+s$, otherwise\footnote{Another good choice is $L_{\tilde k,k}^{s,r}(z) =\qbinom{\tilde k,r} z^{r}$ (see \cite{Korff,GW}). This choice would lead to a similar conclusion.} 
\begin{align}\label{eq:L mx}
L_{\tilde k,k}^{s,r}(z) =\qbinom{k+r,k} z^{r},
\qquad \text{for} \quad k+r=\tilde k+s
\end{align}
where $z$ is called the spectral parameter\footnote{The spectral parameter gives rise to the alphabet in the symmetric functions which are built from $L$.} and we used a convention in \eqref{eq:L mx} which omits a $q$ factor compared to the conventions of \cite{GW}. These matrix elements have a convenient graphical representation in terms of lattice paths on a square tile. In this representation the paths connect bottom and left edges with top and right edges. The labels $(\tilde k,s,k,r)$ of the matrix elements of the $L$ matrix indicate the numbers of paths touching the corresponding edges, for example
\begin{center}
\vspace{-0.5cm}
\begin{align*}
\begin{tikzpicture}[scale=0.3,baseline=(current bounding box.center)]
\draw (0,0) -- (3,0) -- (3,3) -- (0,3) -- (0,0);
\node[left] at (0,1.5) {\footnotesize $s$};
\node[right] at (3,1.5) {\footnotesize $r$};
\node[below] at (1.5,0) {\footnotesize $\tilde k$};
\node[above] at (1.5,3) {\footnotesize $k$};
 \begin{scope}[line width=0.8pt]
\draw (0,1.5+0.3) -- (1.5-0.1,1.5+0.3) -- (1.5-0.1,3);
\draw (0,1.5+0.1) -- (1.5+0.1,1.5+0.1) -- (1.5+0.1,3);;
\draw (0,1.5-0.1) -- (3,1.5-0.1);
\draw (1.5,0) -- (1.5,1.5-0.3) -- (3,1.5-0.3);
\end{scope}
\end{tikzpicture}
\end{align*}
\end{center}
Here we have $(\tilde k,s,k,r)=(1,3,2,2)$. For arbitrary edge occupations we write
\begin{center}
\vspace{-0.5cm}
\begin{align*}
L_{\tilde k,k}^{s,r}=
\begin{tikzpicture}[scale=0.3,baseline=10]
\draw (0,0) -- (3,0) -- (3,3) -- (0,3) -- (0,0);
\node[left] at (0,1.5) {\footnotesize $s$};
\node[right] at (3,1.5) {\footnotesize $r$};
\node[below] at (1.5,0) {\footnotesize $\tilde k$};
\node[above] at (1.5,3) {\footnotesize $k$};
\end{tikzpicture}
\end{align*}
\end{center}
The $L$ matrix naturally acts in a tensor product of two vector spaces whose basis vectors are labelled by integers. If we multiply the $L$ matrices along one of these tensor factors and take the (weighted) trace then we will obtain a transfer matrix $Q$\footnote{We call this transfer matrix the $Q$ matrix because it is expected to play the role of Baxter's  $Q$-operator \cite{Baxter} for an integrable model defined on the relevant physical space \cite{Mangazeev}.}. The matrix elements of the $Q$ matrix can be written as
\begin{align}\label{eq:Q}
Q_{s,r}=
\sum_{k\geq 0}x^k \prod_{i=0}^{n} L^{s_i,r_i}_{k_{i+1},k_{i}}(z_i)
=
\sum_{k\geq 0}x^k\,
\begin{tikzpicture}[scale=0.3,baseline=(current bounding box.center)]
\draw (0,0) -- (2,0) -- (2,2) -- (0,2) -- (0,0);
\draw (0,2+4)  -- (0,0+4) -- (2,0+4) -- (2,2+4) ;
\draw[dashed] (0,2) -- (0,2+2);
\draw[dashed] (2,2) -- (2,2+2);
\draw (0,0+6) -- (2,0+6) -- (2,2+6) -- (0,2+6) -- (0,0+6);
\node[left] at (0,1+6) {\footnotesize $s_0$};
\node[right] at (2,1+6) {\footnotesize $r_0$};
\node[left] at (0,1+4) {\footnotesize $s_1$};
\node[right] at (2,1+4) {\footnotesize $r_1$};
\node[left] at (0,1) {\footnotesize  $s_{n}$};
\node[right] at (2,1) {\footnotesize $r_{n}$};
\node[below] at (1,0) {\footnotesize $k$};
\node[above] at (1,8) {\footnotesize $k$};
\end{tikzpicture}
\end{align}
where 
\begin{align}\label{eq:k's}
k_0=k_{n+1}=k,\qquad 
k_{i}=r_{i-1}+k_{i-1}-s_{i-1},
\qquad 
\text{for } i=1\ldots n
\end{align} 
and $s=(s_0\ldots s_{n})$ and $r=(r_0\ldots r_{n})$. Due to non-vanishing conditions for the matrix elements of $L$ in \eqref{eq:L mx} the integer parameters in $s$ and $r$  need to satisfy
\begin{align}
\sum_{i=0}^n r_i = \sum_{i=0}^n s_i 
\end{align}
The parameter $x$ in \eqref{eq:Q} represents the weight in the trace. This parameter counts the number of paths which wind through the column of $L$ matrices. The matrices $Q$ can be used to construct rectangular partition functions of lattice paths with cylindric boundary conditions (this term is due to the presence of the trace in $Q$). This is achieved by computing products of $Q$ matrices 
\begin{align*}
Z_{s^{(0)},s^{(m)}}=\sum_{s^{(1)}\ldots s^{(m-1)}}Q_{s^{(0)},s^{(1)}}\cdots Q_{s^{(m-1)},s^{(m)}},
\end{align*}
where the summations over $s^{(i)}$ are finite because of the restriction on the labels in \eqref{eq:L mx} and the integer sequences $s^{(0)}=(s_0^{(0)}\ldots s_n^{(0)})$ and $s^{(m)}=(s_0^{(m)}\ldots s_n^{(m)})$ represent the boundary conditions. There exist special choices of boundary conditions for which the quantity $Z_{s^{(0)},s^{(m)}}$ becomes a symmetric polynomial in $z_0\ldots z_n$\footnote{The Macdonald and related symmetric polynomials are usually defined over the field $\mathbb{Q}(q,t)$. Our parameters $x$ and $q$ play the roles of $q$ and $t$ respectively.}. Several known families of symmetric (and non-symmetric) polynomials have been constructed using such cylindric lattice path partition functions \cite{Korff,CdGW,GW,BW,HW}. One limitation of this approach is that currently it cannot explain properties of the symmetric functions such as branching rules and Cauchy identities (see \cite{Macdonald} for background on symmetric functions). In technical terms the reason for this limitation is the presence of the trace in $Q$ in \eqref{eq:Q}. This provides one of the motivations for obtaining formulas for matrix elements of $Q$. 

If we substitute the weights \eqref{eq:L mx} into the $L$ matrices in \eqref{eq:Q} then we will obtain the formula
\begin{align}\label{eq:Q bin}
Q_{s,r}=
z_0^{r_0}\cdots z_n^{r_n} \sum_{k\geq 0}x^k \prod_{i=0}^{n}  
\qbinom{k+r_i+\sum_{j=0}^{i-1}(r_j-s_j),k+\sum_{j=0}^{i-1}(r_j-s_j)}
\end{align}
We identify the indices $s,r$ with the earlier used indices $\alpha$ and $\beta$
\begin{align}\label{eq:alpha beta r s}
\alpha_0=r_0,\qquad 
\beta_i=\sum_{j=0}^{i-1}(r_j-s_j),
\qquad
\alpha_{i}=r_i+\beta_i,
\qquad i=1\ldots n
\end{align}
Assuming that 
\begin{align}\label{eq:conditions x}
\alpha_0\geq \max(\beta), \qquad \beta_i\geq 0,\qquad i=1\ldots n
\end{align}
we can express the matrix elements of $Q$ via $\Phi_{\alpha,\beta}$
\begin{align}\label{eq:Q Phi}
Q_{s,r}=
\frac{\prod_{i=0}^n z_i^{r_i}}{(x;q)_{N+1}} 
\Phi_{\alpha,\beta}(x;q)
\end{align}
where $N=r_0+\cdots +r_n$. If the conditions \eqref{eq:conditions x} do not hold then there exists a reparametrisation of \eqref{eq:alpha beta r s} after which these conditions will be satisfied. This will also produce an overall power of $x$ compared to \eqref{eq:Q Phi} (see \cite{GW} for details). By writing $\Phi_{\alpha,\beta}(x;q)$ using any of the three formulas  \eqref{eq:beta Phi des maj}, \eqref{eq:mstc beta inv} and \eqref{eq:qbinomial formula} we obtain three finite sum expressions for the $Q$ matrix elements. A particularly interesting case is the $q$-binomial formula \eqref{eq:qbinomial formula}, it gives
\begin{align}\label{eq:Q qbinomial formula}
Q_{s,r}=
\frac{\prod_{i=0}^n z_i^{r_i}}{(x;q)_{N+1}} 
\sum_{\kappa}
x^{\kappa_1}
\prod_{i=1}^n
q^{(\kappa_i-\kappa_{i+1})(\beta_i+\kappa_i)}
\qbinom{\alpha_i+\kappa_{i+1},\beta_i+\kappa_i}
\qbinom{\lambda_{i}-\kappa_{i+1},
\lambda_{i}-\kappa_{i}}
\end{align}
It suggests that the $q$-binomial coefficients in \eqref{eq:Q qbinomial formula} can be reexpressed in terms of $L$ matrices of $U_{q}(\widehat{\frak{sl}}_{3})$ type (see \cite{Korff}, \cite{GW}) and thus give rise to a `planar' construction of the $Q$ matrix (i.e. without the trace). On the other hand, using \eqref{eq:qbinomial formula} we showed in Remark \ref{rmk:Phi branching} that $\Phi_{\alpha,\beta}(x;q)$ has a branching formula. We hope that these results would be useful for studying properties of symmetric functions (e.g. branching rules and Cauchy identities) obtained from products of $Q$ matrices such as \eqref{eq:Q}.


\subsection{Outline}
In Section \ref{sec:MacMahon proof} we prove the generalisation of the MacMahon identity stated in Theorem \ref{thm:beta Mac}. In Section \ref{sec:binom mstc} we prove the finite binomial sum formula stated in Theorem \ref{thm:qbinomial} and after that we use this result to prove the generalisation of Carnevale's formula from Theorem \ref{thm:mstc inv}.




\section{The generalised MacMahon identity}\label{sec:MacMahon proof}
We first give a proof of the standard MacMahon identity. For this we use a known proof based on barred multiset permutations \cite{GesselStanley,Moynihan}. We then generalise this proof to the case of Theorem \ref{thm:beta Mac}. The barred multiset permutations are multiset permutations with bars inserted among the letters according to some rules. Such configurations are assigned weights which are designed in a way to reproduce the statistics appearing in the formulas \eqref{eq:Mac identity} and \eqref{eq:beta Mac identity}. In both proofs we count all weighted barred multiset permutations in two different ways which leads to an equation of the form
$$
\sum_{k\geq 0} \cdots = \sum_{w\in S_\delta}\cdots
$$
On the left hand side we fix the number of bars to be $k$ and then show that allocating letters among these bars produces configurations whose weighted sum gives $q$-binomial coefficients. The right hand side of the identity is computed by fixing a word $w$ and placing the bars that are necessarily required by the rules and then computing the weight of that configuration. The latter computation reproduces word statistics by construction.


\subsection{The proof of the MacMahon identity}
In this section we set
\begin{align}
\label{eq:alpha N}
\alpha=(\alpha_1\ldots \alpha_n),\qquad
\alpha_i\geq 0,\qquad  
M= \alpha_1+\cdots + \alpha_n
\end{align}
and $S_\alpha$ is the set of multiset permutations of $\{1^{\alpha_1}\ldots n^{\alpha_n}\}$. We need to prove the identity
\begin{align}\label{eq:Mac identity proof}
\sum_{k\geq 0}x^k \prod_{i=1}^{n}\qbinom{k+\alpha_i,k}
=
\frac{\sum_{w\in S_{\alpha}} x^{\normalfont{\des}(w)} q^{\normalfont{\maj}(w)}}{\prod_{i=0}^{M} (1-xq^i) }
\end{align}
\begin{defn}[Barred words]
\label{def:barred words}
Let $w=(w_1\ldots w_M)\in S_{\alpha}$. A {\it barred} word of $w$ is obtained by placing bars in the $M+1$ positions before, between and after its letters subject to the condition that at least one bar occurs after $w_i$ whenever $i\in\Des(w)$. Equivalently, the bars split $w$ into weakly increasing sequences. Any number of additional bars may be placed in any of the $M+1$ positions. The unique barred configuration minimal under deletion of bars has one bar at each descent of $w$ and no other bars. We call these the \emph{descent bars}. 
\end{defn}
\begin{defn}[Weights of barred words]
\label{def:weights of barred words}
A bar with $i$ letters to its left has weight $xq^i$. The weight of a barred word is the product of the weights of all its bars.
\end{defn}

\begin{ex}
Let $\alpha=(2,1,1,2)$ and $w=(341124)$. The following is a barred word of $w$
\begin{align*}
\setlength{\arraycolsep}{2.5pt}
\begin{array}{r@{\quad}*{13}{c}}
& \nogap
& 3 & \nogap
& 4 & \onebar
& 1 & \nogap
& 1 & \nogap
& 2 & \onebar
& 4 & \twobars
\end{array}
\end{align*}
We indicate the weights of the bars on the same barred word as follows
\begin{align*}
\setlength{\arraycolsep}{2.5pt}
\begin{array}{r@{\quad}*{13}{c}}
& \nogap
& 3 & \nogap
& 4 & \barwt{\onebar}{x q^2}
& 1 & \nogap
& 1 & \nogap
& 2 & \barwt{\onebar}{x q^5}
& 4 & \barwt{\twobars}{(x q^6)^2}
\end{array}
\end{align*}
The weight of this barred word is $x^4q^{19}$.
\end{ex}

\begin{rmk}\label{rmk:Mac RHS from bars}
Definitions \ref{def:barred words} and \ref{def:weights of barred words} immediately give the right hand side of \eqref{eq:Mac identity proof}. Indeed, for a fixed $w\in S_\alpha$, the descent bars contribute $x^{\des(w)}q^{\maj(w)}$. At each position $i=0\ldots M$, additional $j_i$ bars may be inserted independently producing the weight $(x q^i)^{j_i}$. Therefore the weighted sum over all
barred multiset permutations of $w$ is
\begin{align*}
x^{\des(w)}q^{\maj(w)}
\prod_{i=0}^{M}\sum_{j_i\geq 0}(xq^i)^{j_i}
=
\frac{x^{\des(w)}q^{\maj(w)}}
{\prod_{i=0}^{M}(1-xq^i)}
\end{align*}
Summing over $w\in S_\alpha$ gives the required right hand side of \eqref{eq:Mac identity proof}.
\end{rmk}

It remains to recover the left hand side of \eqref{eq:Mac identity proof} by fixing the total number of bars. Our computation below is an adaptation of the proof of  \cite{Moynihan}.

\begin{prop}\label{prop:fixed bars Mac}
Fix $k\geq 0$, the sum of the weights of all barred multiset permutations of all words $w\in S_\alpha$ with $k$ bars is
\begin{align}\label{eq:k-bars weight}
x^k\prod_{i=1}^{n}\qbinom{k+\alpha_i,k}
\end{align}
\end{prop}

\begin{proof}
We can form all barred words of $w\in S_\alpha$ with $k$ bars recursively. We start from an empty word and place $k$ bars, then we consider all locations where we can insert letters $n$, then all locations where we can insert letters $n-1$ etc. By placing $k$ bars we form $k+1$ compartments where letters can be inserted, for example 
\begin{align*}
\setlength{\arraycolsep}{2.5pt}
\begin{array}{r@{\quad}*{13}{c}}
&  & \barwt{\onebar}{x}
&  & \barwt{\onebar}{x}
&  & \barwt{\onebar}{x}
&  & \barwt{\onebar}{x} 
\end{array}
\end{align*}
The associated weight of the bars is $x^k$. We number these compartments from right to left starting from $0$.
Then we place $\alpha_n$ copies of the letter $n$ and count the associated weight. For example
\begin{align*}
\setlength{\arraycolsep}{2.5pt}
\begin{array}{r@{\quad}*{13}{c}}
&  & \barwt{\onebar}{x}
&   n n n & \barwt{\onebar}{x   q^3}
&   & \barwt{\onebar}{x  q^3}
&   n & \barwt{\onebar}{x   q^4}
&   n n &
\end{array}
\end{align*}
In this way we are placing the first letter $n$ in the $r_1$-th compartment, the second letter $n$ in the $r_2$-th compartment and so on. Such a distribution is encoded by the weakly increasing sequence $r_1\leq \cdots \leq r_{\alpha_n}$. The $i$-th letter $n$ placed in the $r_i$-th compartment gives the weight $q$ to each of the $r_i$ bars to its right contributing the total weight of $q^{r_i}$. 
Hence the weighted sum over all possible configurations with $\alpha_n$ letters $n$ is given by
\begin{align*}
\sum_{0\leq r_1\leq\cdots\leq r_{\alpha_n}\leq k}
q^{r_1+\cdots+r_{\alpha_n}}
 =\qbinom{k+\alpha_n,k}
\end{align*}
In the next step we perform the same procedure with $\alpha_{n-1}$ copies of the letter $n-1$ and count the associated weight. It is important to note that, since all letters between two consecutive bars must form a weakly increasing sequence, the letters $n-1$ within a compartment must be placed in a fixed location: before all letters $n$. For example
\begin{align*}
\setlength{\arraycolsep}{2.5pt}
\begin{array}{r@{\quad}*{13}{c}}
&  (n-1) (n-1) & \barwt{\onebar}{x q^2 }
&  (n-1) n n n & \barwt{\onebar}{x   q^6}
&  & \barwt{\onebar}{x  q^6}
&  n & \barwt{\onebar}{x   q^7}
&  (n-1) n n &
\end{array}
\end{align*}
This shows that at the $(n-i+1)$-th step  if we need to add $\alpha_i$ letters $i$ to a barred word then once we determine in which compartments they are placed then their precise locations within compartments are fixed. Hence the contribution from adding $\alpha_i$ letters $i$ in all possible locations is given by 
\begin{align*}
\sum_{0\leq r_1\leq\cdots\leq r_{\alpha_i}\leq k}
q^{r_1+\cdots+r_{\alpha_i}}
 =\qbinom{k+\alpha_i,k}
\end{align*}
Collecting the weights from all ways of distributing letters $(1^{\alpha_1}\ldots n^{\alpha_n})$ over all compartments gives us the weight
\begin{align*}
x^k \prod_{i=1}^{n}\qbinom{k+\alpha_i,k}
\end{align*}
\end{proof}
Summing \eqref{eq:k-bars weight} over $k$ gives the weighted sum over all barred configurations for all words in $S_\alpha$ and matches with the left hand side of \eqref{eq:Mac identity proof}. Combining this with Remark \ref{rmk:Mac RHS from bars} proves the identity \eqref{eq:Mac identity proof}.


\subsection{The proof of the generalised MacMahon identity (Theorem \ref{thm:beta Mac})}
In this section we set
\begin{align}
\label{eq:alpha beta delta}
\alpha=(\alpha_0,\alpha_1\ldots \alpha_n),\qquad 
\beta=(\beta_0,\beta_1\ldots \beta_n), \qquad 
\delta=\alpha-\beta,
\qquad N=\delta_0+\cdots +\delta_n
\end{align}
and $S_\delta$ is the set of multiset permutations of $\{0^{\delta_0}\ldots n^{\delta_n}\}$.  We impose the following conditions on the integer sequences $\alpha$ and $\beta$ 
\begin{align}
\label{eq:beta conditions}
&\beta_0=0,\qquad
\beta_1\geq \cdots\geq \beta_n\geq 0\\
\label{eq:alpha beta conditions}
&\alpha_0 \geq \beta_1,\qquad
\alpha_i\geq\beta_i\quad (i=1,\ldots,n)
\end{align}
We need to prove the identity
\begin{align}\label{eq:beta Mac identity proof}
\sum_{k\geq 0}x^k \qbinom{k+\alpha_0,k}\prod_{i=1}^{n}\qbinom{k+\alpha_i,k+\beta_i}
=
\frac{\sum_{w\in S_{\delta}} x^{\normalfont{\des}_\beta(w)} q^{\normalfont{\maj}_\beta(w)+\inv_\beta(w)}}{\prod_{i=0}^{N} (1-x q^i) }
\end{align}
\begin{defn}[$\beta$-barred words]
\label{def:beta barred word}
Let $w=(w_1\ldots w_N)\in S_{\delta}$. A {\it $\beta$-barred} word of $w$ is obtained by placing bars in the $N+1$ positions before, between and after its letters subject to the condition that at least one bar occurs after $w_i$ whenever $i\in\Des_\beta(w)$. Any number of additional bars may be placed in any of the $N+1$ positions. The necessary bars located after $w_i$ for $i\in\Des_\beta(w)$ are called the {\it $\beta$-descent bars}. 
\end{defn}
\begin{defn}[Weights of $\beta$-barred words]
\label{def:weights of beta barred word}
A bar with $i$ letters to its left has weight $xq^i$. In addition to this every letter $0\in w$ appearing in position $j$ has the weight $q^{m_j}$ where $m_j$ counts the number of $\widehat w_i>\widehat w_j$ for $i<j $. From the definition of $\widehat w$ it follows that  only the first $\beta_1$ letters $0\in w$ will contribute a non-trivial weight. The weight of a $\beta$-barred word is the product of all its bar and zero weights. 
\end{defn}

In examples we put $w$ above $\widehat w$ and draw each bar vertically through both words. We indicate the bar weights under the corresponding bars and zero weights above the corresponding zeros of $w$.
\begin{ex}
Let $(\beta_1,\beta_2,\beta_3)=(3,1,0)$ and $w=(123100200)$. Then $\widehat w=(310301123)$ and $\Des_\beta(w)=\{3,7\}$. The $\beta$-barred word of $w$ containing only its $\beta$-descent bars is
\begin{equation*}
\barreddoubleword
  {8}
  {1,2,3,1,0,0,2,0,0}
  {3,1,0,3,0,1,1,2,3}
  {3/1,7/1}
  {}
\end{equation*}
The bar and zero weights of this $\beta$-barred word are indicated as follows
\begin{equation*}
\weighteddoubleword
  {8}
  {1,2,3,1,0,0,2,0,0}
  {3,1,0,3,0,1,1,2,3}
  {3/1,7/1}
  {3,7}
  {5/q^3,6/q^2,8/q^2,9/1}
  {}
\end{equation*}
The first three zeros have weights $q^3$, $q^2$ and $q^2$ and the last one has a trivial weight. This happens because the $j$-th $0$ in $w$ will not create an inversion for the corresponding letter in $\widehat w$ when $j> \beta_1$ (as mentioned in Definition \ref{def:weights of beta barred word}). The combined zero weight equals $q^7$ which coincides with $q^{\inv_\beta(w)}$. This $\beta$-barred word has weight $x^2q^{17}$.
\end{ex}

\begin{rmk}\label{rmk:beta Mac RHS from bars}
Fix a word  $w\in S_\delta$. The product of the zero weights of a barred word of $w$ equals $q^{\inv_\beta(w)}$. The weighted sum over all $\beta$-barred words of $w$ is
\begin{align}
\label{eq:beta-weight of w}
x^{\des_\beta(w)}q^{\maj_\beta(w)+\inv_\beta(w)}
\prod_{i=0}^{N}\sum_{j_i\geq 0}(xq^i)^{j_i}
=
\frac{x^{\des_\beta(w)}q^{\maj_\beta(w)+\inv_\beta(w)}}
{\prod_{i=0}^{N}(1-xq^i)}
\end{align}
\end{rmk}
\begin{proof}
We recall that $\Inv_\beta(w)=\Inv(w)\cap \Inv(\widehat w)$. Consider a pair of indices $(i,j)$ with $i<j$. Suppose $w_i>w_j$ and  $w_j>0$, then the pair $(i,j)\in \Inv(w)$ by definition. However, $(i,j)\notin \Inv(\widehat w)$ because $\widehat w_i=\beta_{w_i}\leq \beta_{w_j}= \widehat w_j$ (recall the imposed ordering $\beta_1\geq\cdots\geq\beta_n$). This means that in order to find elements of $\Inv_\beta(w)$ we must look for the pairs $(i,j)$ with $w_i>w_j$ and $w_j=0$. Such pairs $(i,j)$ can be counted by running through all instances $w_j=0$ and  counting the numbers of $\widehat w_i>\widehat w_j$ for $i<j$. This counting gives us $\inv_\beta(w)$ producing the weight $q^{\inv_\beta(w)}$ which by Definition \ref{def:weights of beta barred word} is equal to the product of the zero weights of $w$. 

By Definition \ref{def:weights of beta barred word}, for a fixed $w\in S_\delta$, the $\beta$-descent bars contribute the weight $x^{\des_\beta(w)}q^{\maj_\beta(w)}$ while the zero weights contribute $q^{\inv_\beta(w)}$. Since the additional $j_i$ bars may be inserted independently at every position $i=0\ldots N$, computing the weighted sum over all barred words of $w$ gives \eqref{eq:beta-weight of w}. 
\end{proof}

By summing \eqref{eq:beta-weight of w} over $w$ we obtain the right hand side of \eqref{eq:beta Mac identity proof}. In the following proposition we compute the weighted sum of $\beta$-barred words with a fixed number of bars. This will allow us to obtain the left hand side of \eqref{eq:beta Mac identity proof}.

\begin{prop}\label{prop:fixed bars beta Mac}
For every $k\geq 0$, the sum of the weights of all $\beta$-barred words of all $w\in S_\delta$ with exactly $k$ bars is
\begin{align}\label{eq:k-bars beta weight}
x^k\qbinom{k+\alpha_0,k}
\prod_{i=1}^{n}\qbinom{k+\alpha_i,k+\beta_i}
\end{align}
\end{prop}
\begin{proof}
As in the proof of Proposition \ref{prop:fixed bars Mac}, we construct the barred words by fixing $k$ bars and adding letters of $w$ recursively. Begin with $k$ bars and no letters, for example
\begin{align*}
\setlength{\arraycolsep}{2.5pt}
\begin{array}{r@{\quad}*{13}{c}}
&   & \barwt{\onebar}{x}
&   & \barwt{\onebar}{x}
&   & \barwt{\onebar}{x}
&   & \barwt{\onebar}{x} 
\end{array}
\end{align*}
The bars give the weight $x^k$. The bars create $k+1$ compartments which we number from right to left starting from $0$. We insert the $\delta_0=\alpha_0$ letters $0$, for example
\begin{equation}\label{eq:proof barred word example 0}
\manualweighteddoubleword
  {5}
  {0,0,0, ,0}
  {0,1,2, ,3}
  {0/1,1/1,3/1,4/1}
  {0/q^0,1/q^1,3/q^3,4/q^3}
  {1/q^0,2/q^0,3/q^0,5/q^0}
  {}
\end{equation}
By the same logic as in Proposition \ref{prop:fixed bars Mac} the weighted sum over all possible configurations with $\alpha_0$ letters $0$ is given by
\begin{align*}
\sum_{0\leq r_1\leq\cdots\leq r_{\alpha_0}\leq k}
q^{r_1+\cdots+r_{\alpha_0}}
 =\qbinom{k+\alpha_0,k}
\end{align*}
In the next step we need to insert $\delta_n$ letters $n$ in all possible ways which produce valid barred words. In order to compute the weighted sum over these configurations we will use the formula 
\begin{align}\label{eq:binom alpha-beta}
\sum_{0\leq r_1\leq\cdots\leq r_{\delta_n}\leq k+\beta_n}
q^{r_1+\cdots+r_{\delta_n}}
 =\qbinom{k+\alpha_n,k+\beta_n}
\end{align}
This expression tells us that we are distributing $\delta_n$ letters $n$ among $k+\beta_n+1$ compartments and each $i$-th letter contributes the weight $q^{r_i}$. This increase in the number of compartments can be achieved by adding extra $\beta_n$ bars. Indeed, the zero weights depend on how $\delta_n$ letters $n$ are inserted with respect to the first $\beta_n$ $0$'s. Hence the first $\beta_n$ $0$'s effectively form another set of bars which we call $0$-bars. The $0$-bars split the existing compartments further. We note that the $0$-bars carry no $x$ weight. For example, if $\beta_n=2$ we may draw the $0$-bars through the first $2$ $0$'s
\begin{equation}\label{eq:proof barred word example 0 bars}
\manualweighteddoubleword
  {5}
  {0,0,0, ,0}
  {0,1,2, ,3}
  {0/1,1/1,3/1,4/1}
  {0/q^0,1/q^1,3/q^3,4/q^3}
  {1/q^0,2/q^0,3/q^0,5/q^0}
  {1/1,2/1}
\end{equation}
Then we may insert $\delta_n$ letters $n$ in $k+\beta_n+1$ compartments, for example  
\begin{equation}\label{eq:proof barred word example 0 bars and n's}
\manualweighteddoubleword
  {5}
  {n,n,0,n,0,0,n,n, ,0,n}
  {2,2,0,2,1,2,2,2, ,3,2}
  {1/1,4/1,8/1,9/1}
  {1/q,4/q^4,8/q^8,9/q^8}
  {3/q^2,5/q^3}
  {3/1,5/1}
\end{equation}
In these $k+\beta_n+1$ compartments the uncrossed zeros precede the letters $n$. Thus the letters of $w$ are weakly increasing within every compartment. A copy of $n$ placed in the $r$-th compartment (counted from the right) contributes the factor $q^r$, i.e. one factor $q$ to every ordinary bar to its right and one factor $q$ to the zero weight at every $0$-bar to its right. Such configurations are counted precisely by \eqref{eq:binom alpha-beta}.

At each subsequent step $i=n-1,n-2,\ldots,1$, suppose that the letters $n,n-1,\ldots,i+1$ have already been inserted. Since $\beta_i\geq\beta_{i+1}$, we add $\beta_i-\beta_{i+1}$ new $0$-bars so that the first $\beta_i$ zeros now form the $0$-bars. Thus adding further letters also produces further $0$-bars and subdivides the existing compartments. Together with the $k$ ordinary bars the $0$-bars form $k+\beta_i+1$ compartments numbered from right to left. Within each compartment the uncrossed zeros come first, followed by the positive letters in weakly increasing order. Hence the $\delta_i$ copies of $i$ have a unique possible location in each compartment: after the uncrossed zeros and before the already inserted letters. A copy of $i$ placed in the $r$-th compartment (counted from the right) contributes the factor $q^r$, i.e. one factor $q$ to every ordinary bar to its right and one factor $q$ to the zero weight at every $0$-bar to its right. Therefore the weighted sum at this step is
\begin{align*}
\sum_{0\leq r_1\leq\cdots\leq r_{\delta_i}\leq k+\beta_i}
q^{r_1+\cdots+r_{\delta_i}}
=\qbinom{k+\alpha_i,k+\beta_i}
\end{align*}
Repeating this step down to $i=1$ produces every valid $\beta$-barred word with $k$ bars uniquely. Multiplying the contributions from the bars, the zeros and all positive letters gives
\begin{align*}
x^k\qbinom{k+\alpha_0,k}
\prod_{i=1}^{n}\qbinom{k+\alpha_i,k+\beta_i}
\end{align*}
as required.
\end{proof}
Summing \eqref{eq:k-bars beta weight} over $k$ gives the weighted sum over all $\beta$-barred configurations of all words of $S_\delta$ and matches with the left hand side of \eqref{eq:beta Mac identity proof}. Combining this with Remark \ref{rmk:beta Mac RHS from bars} proves the identity \eqref{eq:beta Mac identity proof}.




\section{The binomial formula and the $(\mstc_\beta,\inv)$-formula}\label{sec:binom mstc}
In this section we prove Theorem \ref{thm:qbinomial} and after that we use it to prove Theorem \ref{thm:mstc inv}.  As before we let $\alpha=(\alpha_0,\alpha_1\ldots \alpha_n)$ and $\beta=(\beta_1\ldots \beta_n)$ be two collections of integers satisfying
\begin{align}
\label{eq:condition alpha beta X}
\alpha_0\geq \max(\beta), 
\qquad 
\alpha_i\geq \beta_{i}\geq 0
\end{align}
for $i=1\ldots n$. We set $\gamma_i=\alpha_i-\beta_i$ for $i=1\ldots n$ and $\gamma_{n+1}=\alpha_0$ and introduce the integers
\begin{align}\label{eq:lambda X}
\lambda_i=\gamma_{i+1}+\cdots+\gamma_{n+1}-\beta_i,\qquad i=1\ldots n
\end{align}

\subsection{Proof of Theorem \ref{thm:qbinomial}}\label{sec:proof binomial}
Recall the definition of the normalised series $\Phi_{\alpha,\beta}(x;q)$ and for any $m=1\ldots  n+1$ and fixed $\alpha$ and $\beta$  define a related family of functions
\begin{align}\label{eq:Psi m}
\Psi^{(m)}(x)=(x;q)_{\sum_{i=m}^{n+1}\gamma_i+1}\sum_{k\geq 0} x^k\qbinom{k+\alpha_0,k}
\prod_{i=m}^{n}\qbinom{k+\alpha_i,k+\beta_i}
\end{align}
The two extreme cases are
\begin{align}\label{eq:Psi 0 n}
\Psi^{(n+1)}(x)=(x;q)_{\alpha_0+1}\sum_{k\geq 0} x^k\qbinom{k+\alpha_0,k}=1
\qquad \text{and} \qquad
\Psi^{(1)}(x)=\Phi_{\alpha,\beta}(x;q)
\end{align}
We can expand the functions $\Psi^{(m)}(x)$ in $x$ and write
\begin{align}\label{eq:Psi m coefficients}
\Psi^{(m)}(x)= \sum_{k\geq 0} \psi^{(m)}_{k} x^k
\end{align}
where $\psi^{(m)}_{k}$ vanish for large $k$ as we will prove below independently from our earlier result. The purpose of introducing the family $\Psi^{(m)}$ is to use a relation between the coefficients $\psi^{(m)}_{k}$ and $\psi^{(m+1)}_{k}$ in order to prove Theorem \ref{thm:qbinomial}. This relation is provided by the following Lemma.
\begin{lem}\label{lem:coefficient transformation}
Let $a,b,c$ be non-negative integers such that $a,c\geq b$. Consider two expressions
\begin{align}
F(x)
&=(x;q)_{c+1}\sum_{k\geq 0}x^ku_k
=\sum_{\ell\geq0}f_\ell x^\ell\label{eq: F series}\\
\widetilde F(x)
&=(x;q)_{c+a-b+1}\sum_{k\geq 0}x^k
\qbinom{k+a,k+b}u_k
=\sum_{\ell\geq0}\widetilde f_\ell x^\ell\label{eq: F series tilde}
\end{align}
Suppose that $F(x)$ is a polynomial of degree at most $c-b$ in $x$. Then
\begin{align}\label{eq:coefficient transformation}
\widetilde f_\ell
=
\sum_{j=0}^\ell
q^{(\ell-j)(b+\ell)}
\qbinom{a+j,b+\ell}
\qbinom{c-b-j,\ell-j}
f_{j}
\end{align}
and  the degree of $\widetilde F(x)$ is at most  $\deg F(x)+a-b$.
\end{lem}

\begin{proof}
We consider two functions 
\begin{align}\label{eq:normalised Fs}
\frac{1}{(x;q)_{c+1}}F(x) , \qquad 
\frac{1}{(x;q)_{c+a-b+1}}\widetilde F(x) 
\end{align}
in which we expand the inverse Pochhammer symbols using the $q$-binomial series
\begin{align}
\label{eq:q-binom series}
\frac{1}{(x;q)_{j+1}}= \sum_{k\geq 0}x^k \qbinom{j+k,k}
\end{align}
Using \eqref{eq:normalised Fs}, \eqref{eq: F series} and \eqref{eq:q-binom series} we can express $u_k$ as a linear combination of $f_\ell$ and using \eqref{eq:normalised Fs}, \eqref{eq: F series tilde} and \eqref{eq:q-binom series} we can express $u_k$ as a linear combination of $\widetilde f_\ell$
\begin{align}\label{eq:u from f}
u_k
&=
\sum_{\ell=0}^k
\qbinom{c+k-\ell,k-\ell}f_\ell \\
\label{eq:u from f tilde}
\qbinom{k+a,k+b}u_k
&=
\sum_{\ell=0}^k
\qbinom{a+c-b+k-\ell,k-\ell}\widetilde f_\ell
\end{align}
These two equations lead to a relation between $f_\ell$ and $\widetilde f_\ell$
\begin{align}\label{eq:f ftilde eq}
\sum_{\ell=0}^k
\qbinom{a+c-b+k-\ell,k-\ell}\widetilde f_\ell=
\sum_{\ell=0}^k
\qbinom{k+a,k+b}\qbinom{c+k-\ell,k-\ell}f_\ell
\end{align}
This identity can be used to express $\widetilde f_\ell$ in terms of $f_\ell$. For this we can interpret the coefficients in front of $\widetilde f_\ell$ and $f_\ell$ as invertible lower triangular matrices 
\begin{align*}
X_{k\ell}
=
\qbinom{c+a-b+k-\ell,k-\ell},\qquad 
Y_{k\ell}
=
\qbinom{k+a,k+b}
\qbinom{c+k-\ell,k-\ell}
\end{align*}
for $0\leq k,\ell\leq K$ where $K$ is a sufficiently large positive integer. Then 
\begin{align*}
\widetilde f_\ell
=
\sum_{j=0}^\ell
(X^{-1}Y)_{\ell j}
f_{j}
\end{align*}
Inversion of $X$ follows from a standard manipulation with $q$-binomial coefficients while $X^{-1}Y$ can be computed using the terminating $q$-Pfaff--Saalsch\"utz summation \cite[Equation (II.12)]{GasperRahman}. This gives
\begin{align*}
(X^{-1}Y)_{\ell j}
=
q^{(\ell-j)(b+\ell)}
\qbinom{a+j,b+\ell}
\qbinom{c-b-j,\ell-j},
\qquad 0\leq j\leq c-b
\end{align*}
This is precisely the coefficient appearing in \eqref{eq:coefficient transformation}. Note that we must have $\ell-j \leq a-b$ in order for the first binomial coefficient to be non-vanishing. This implies that the degree of $\widetilde F(x)$ cannot increase by more than $a-b$ compared to the degree of $F(x)$. This completes the proof. 
\end{proof}
Let us use this Lemma to complete the proof of Theorem \ref{thm:qbinomial}. Consider the following choice of $F(x)$ and $\widetilde F(x)$
\begin{align*}
F(x)=\Psi^{(m+1)}(x),
\qquad 
\widetilde F(x)=\Psi^{(m)}(x)
\end{align*}
By \eqref{eq:Psi m} this pair has the correct forms of \eqref{eq: F series} and \eqref{eq: F series tilde} with the parameters
\begin{align*}
c =\lambda_{m}+\beta_{m},\qquad
a=\alpha_m,\qquad
b=\beta_m
\end{align*}
We first show that $\Psi^{(m)}(x)$ is a polynomial that satisfies the condition
\begin{align}\label{eq:Psi deg}
\deg\Psi^{(m)} \leq\sum_{i=m}^{n}\gamma_i 
\end{align}
For $m=n+1$ we have $\Psi^{(n+1)}(x)=1$ hence the above condition holds. Suppose it also holds for $\Psi^{(m+1)}(x)$. Since $\alpha_0\geq \beta_m$ we have
\begin{align*}
\deg\Psi^{(m+1)}\leq 
\sum_{i=m+1}^{n}\gamma_i= \lambda_m+\beta_m-\alpha_0\leq \lambda_m
\end{align*}
This is the required degree condition on $\Psi^{(m+1)}$ in the Lemma. Hence we can apply the Lemma which guarantees 
\begin{align*}
\deg\Psi^{(m)}\leq \deg\Psi^{(m+1)}+\alpha_m-\beta_m
\leq\sum_{i=m}^{n}\gamma_i
\end{align*}
This verifies \eqref{eq:Psi deg}

Next we can use \eqref{eq:coefficient transformation} to write down a relation between the coefficients $\psi_k^{(m)}$ and $\psi_{k'}^{(m+1)}$
\begin{align}\label{eq:Psi coefficient recursion}
\psi^{(m)}_{k}
=\sum_{k'=0}^{k}
q^{(k-k')(\beta_m+k)}
\qbinom{\alpha_m+k',\beta_m+k} \qbinom{\lambda_m-k',
\lambda_m-k}
\psi^{(m+1)}_{k'}
\end{align}
Set the labels $k=\kappa_m$ and $k'=\kappa_{m+1}$ and iterate this recurrence relation using the initial condition $\psi^{(n+1)}_0=1$. This gives
\begin{align}\label{eq:psi m}
\psi^{(m)}_{\kappa_m}
=
\sum_{\kappa_{m+1}\ldots\kappa_{n}\geq0}
\prod_{i=m}^n
q^{(\kappa_i-\kappa_{i+1})(\beta_i+\kappa_i)}
\qbinom{\alpha_i+\kappa_{i+1},\beta_i+\kappa_i}
\qbinom{\lambda_i-\kappa_{i+1},
\lambda_i-\kappa_i}
\end{align}
where $\kappa_{n+1}=0$. We compute $\Psi^{(1)}(x)=\Phi_{\alpha,\beta}(x;q)$ by inserting \eqref{eq:psi m} into \eqref{eq:Psi m coefficients} with $m=1$
\begin{align*}
\Psi^{(1)}(x)= \sum_{\kappa_1\geq 0} \psi^{(1)}_{\kappa_1} x^{\kappa_1}
=
\sum_{\kappa_1\ldots\kappa_{n}\geq0}
x^{\kappa_1}
\prod_{i=1}^n
q^{(\kappa_i-\kappa_{i+1})(\beta_i+\kappa_i)}
\qbinom{\alpha_i+\kappa_{i+1},\beta_i+\kappa_i}
\qbinom{\lambda_i-\kappa_{i+1},
\lambda_i-\kappa_i}
\end{align*}
This completes the proof of Theorem \ref{thm:qbinomial}.

\begin{rmk}
\label{rmk:Phi branching}
The finite binomial formula \eqref{eq:qbinomial formula}, proved
above can be used to derive a branching formula for $\Phi_{\alpha,\beta}$. For an integer $r$ define
\begin{align}
\label{eq:branching parameters}
\alpha^{(r)}=(\alpha_0+\alpha_n-\beta_n, \alpha_1+r\ldots\alpha_{n-1}+r),
\qquad
\beta^{(r)}=(\beta_1+r\ldots \beta_{n-1}+r)
\end{align}
Then
\begin{align}
\label{eq:Phi branching}
\Phi_{\alpha,\beta}(x;q)=\sum_{r=0}^{\min(\alpha_0,\alpha_n)-\beta_n}x^r q^{r(\beta_n+r)}
\qbinom{\alpha_n,\beta_n+r}\qbinom{\alpha_0-\beta_n,r} \Phi_{\alpha^{(r)},\beta^{(r)}}(x;q)
\end{align}
\end{rmk}
\begin{proof}
We group the summands in \eqref{eq:qbinomial formula}
according to the value $\kappa_n=r$. Since
$\kappa_{n+1}=0$ and $\lambda_n=\alpha_0-\beta_n$, the factor
with $i=n$ is
\begin{align*}
q^{r(\beta_n+r)}
\qbinom{\alpha_n,\beta_n+r}
\qbinom{\alpha_0-\beta_n,r}
\end{align*}
In the remaining factors we set
\begin{align*}
\kappa_i=r+\nu_i,
\qquad \text{for } 1\leq i<n,
\qquad
\nu_n=0
\end{align*}
and take $\nu_1\ldots \nu_{n-1}$ to be the new summation indices.  The parameters of $(\alpha^{(r)},\beta^{(r)})$ satisfy
\begin{align*}
\alpha_i^{(r)}-\beta_i^{(r)}
=\alpha_i-\beta_i,
\qquad
\lambda_i^{(r)}:=
\alpha_0^{(r)}
+\sum_{j=i+1}^{n-1}
  \bigl(\alpha_j^{(r)}-\beta_j^{(r)}\bigr)
-\beta_i^{(r)}
=\lambda_i-r
\end{align*}
for $i=1\ldots n-1$. After this change of summation indices the product of the $q$-binomial factors with $i<n$ is exactly the product appearing in the binomial formula for $\Phi_{\alpha^{(r)},\beta^{(r)}}(x;q)$. This proves \eqref{eq:Phi branching}.

The identity \eqref{eq:Phi branching} can be written in terms of the basic hypergeometric series. This formula was stated in the introduction \eqref{eq:hgf branching intro}
\begin{align}
\label{eq:hgf branching}
\phi_{\alpha,\beta}(x;q)=\sum_{r\geq 0}(xq^{\alpha_0^{(r)}+1})^r
\frac{(q^{\beta_n-\alpha_0};q)_r (q^{\beta_n-\alpha_n};q)_r}{(q;q)_r(q^{\beta_n+1};q)_r}\prod_{i=1}^{n-1}\frac{(q^{\alpha_i+1};q)_r}
     {(q^{\beta_i+1};q)_r}
\phi_{\alpha^{(r)},\beta^{(r)}}(x;q)
\end{align}
The sum over $r$ in this formula truncates at $r=\min(\alpha_0,\alpha_n)-\beta_n$ due to the $q$-Pochhammer symbols which contain non-positive powers of $q$. This equation can be obtained by applying an identity from \cite{SrivastavaJain} which involves the hypergeometric series with generic parameters and an infinite summation.
\end{proof}


\subsection{Proof of Theorem \ref{thm:mstc inv}}
In this section we give a proof of the formula
\begin{align}
\label{eq:Phi mstc inv proof}
\Phi_{\alpha,\beta}(x;q)
=
 \sum_{w\in S_\gamma} x^{\normalfont{\mstc}_\beta(w)} q^{\normalfont{\inv}(w)}
\end{align}
Our proof is based on results of Section \ref{sec:proof binomial}. More specifically we will use the recurrence relation \eqref{eq:Psi coefficient recursion} in order to prove that \eqref{eq:Phi mstc inv proof} holds for all $\Psi^{(m)}(x)$ hence our proof of \eqref{eq:Phi mstc inv proof} is inductive. 

For $1\leq m\leq n+1$ let $\gamma^{(m)}$ be the composition obtained from $\gamma$ by setting the multiplicities of labels $1\ldots m-1$ to zero. Thus $S_{\gamma^{(m)}}$ consists of multiset permutations of
\begin{align*}
\{m^{\gamma_m}\ldots(n+1)^{\gamma_{n+1}}\}
\end{align*}
Similarly to Definition \ref{def:mstc beta} we define
\begin{align*}
\beta^{\gamma^{(m)}}=(\beta_m^{\gamma_m}\ldots\beta_n^{\gamma_n},0^{\gamma_{n+1}})
\end{align*}
We need to show that
\begin{align}\label{eq:Psi word sum}
\Psi^{(m)}(x)=\sum_{w\in S_{\gamma^{(m)}}}x^{\mstc_\beta(w)}q^{\inv(w)}
\end{align}
For $m=n+1$ we have $\gamma^{(n+1)}=(0\ldots 0,\alpha_0)$ and $S_{\gamma^{(n+1)}}$ contains a single word $(n+1)^{\alpha_0}$. Thus \eqref{eq:Psi word sum} gives
\begin{align}\label{eq:Psi word sum m=n+1}
\Psi^{(n+1)}(x)=1
\end{align}
We will derive \eqref{eq:Psi word sum} assuming that
\begin{align}\label{eq:Psi word sum m+1}
\Psi^{(m+1)}(x)=\sum_{v\in S_{\gamma^{(m+1)}}}x^{\mstc_\beta(v)}\, q^{\inv(v)}
\end{align}
holds. We can rewrite \eqref{eq:Psi word sum m+1} as
\begin{align}\label{eq:Psi m+1 coefs}
\Psi^{(m+1)}(x)= \sum_{k'\geq 0} \psi^{(m+1)}_{k'} x^{k'}=
\sum_{k'\geq 0}x^{k'}
\sum_{\substack{v\in S_{\gamma^{(m+1)}}\\ \mstc_\beta(v)=k' } }q^{\inv(v)}
\end{align}
This gives us the coefficient $\psi^{(m+1)}_{k'}$ in the form
\begin{align}\label{eq:psi m+1 formula}
\psi^{(m+1)}_{k'} =
\sum_{\substack{v\in S_{\gamma^{(m+1)}}\\ \mstc_\beta(v)=k' } } q^{\inv(v)}
\end{align}
Insert this into \eqref{eq:Psi coefficient recursion}
\begin{align}\label{eq:psi coeff X}
\psi^{(m)}_{k}
=
\sum_{k'=0}^{k}
\sum_{\substack{v\in S_{\gamma^{(m+1)}}\\ \mstc_\beta(v)=k' } }
 q^{\inv(v)+(\beta_m+k)(k-k')}
\qbinom{\alpha_m+k',\beta_m+k} \qbinom{\lambda_m-k',
k-k'}
\end{align}
where we rewrote the lower argument of the second binomial coefficient using its basic property. We note that the $q$-binomial coefficients vanish unless the following conditions are satisfied
\begin{align}
\label{eq:k k' range}
k\leq \lambda_m, \qquad k-k'\leq \gamma_m
\end{align}
We would like to rewrite the formula \eqref{eq:psi coeff X} for $\psi^{(m)}_{k}$ in the form \eqref{eq:psi m+1 formula}. Consider two permutations 
\begin{align}\label{eq:rho sigma}
\rho\in S_{\alpha_m+k'}/(S_{\beta_m+k}\times S_{\gamma_m-k+k'}),
\qquad 
\sigma \in S_{\lambda_m-k'}/(S_{\lambda_m-k}\times S_{k-k'})
\end{align}
Fix $v\in S_{\gamma^{(m+1)}}$ with $\mstc_\beta(v)=k'$ and consider a word $w\in S_{\gamma^{(m)}}$ with the following structure
\begin{align}\label{eq:w}
w=(\rho(v_1\ldots v_{\beta_m+k}\,m^{\gamma_m-k+k'}),\sigma(v_{\beta_m+k+1}\ldots v_{\beta_m+\lambda_m}\,m^{k-k'}))
\end{align} 
The two permutations $\rho$ and $\sigma$ are chosen in \eqref{eq:rho sigma} in such a way that they do not change the order of the letters of $v$ in the word $w$. 
The following formulas hold 
\begin{align}\label{eq:binom inv w}
q^{\inv(v)+(\beta_m+k)(k-k')}
\qbinom{\alpha_m+k',\beta_m+k} \qbinom{\lambda_m-k',
k-k'}=\sum_{\rho,\sigma} q^{\inv(w)}
\end{align}
and 
\begin{align}\label{eq:mstc beta w}
\normalfont{\mstc}_\beta(w)=k
\end{align}
These two formulas are proved in Lemma \ref{lem:w} and Lemma \ref{lem:mstc =k} respectively. Using \eqref{eq:binom inv w} in \eqref{eq:psi coeff X}  gives
\begin{align}\label{eq:psi coeff Y}
\psi^{(m)}_{k}
=
\sum_{k'=\max(0,k-\gamma_m)}^{k}
\sum_{\substack{v\in S_{\gamma^{(m+1)}}\\ \mstc_\beta(v)=k' } }
\sum_{\rho,\sigma}
q^{\inv(w)}
\end{align}
where $w$ is of the form \eqref{eq:w}. Let us rewrite the sums in \eqref{eq:psi coeff Y} in terms of a sum over $S_{\gamma^{(m)}}$.  

Consider a word $w\in S_{\gamma^{(m)}}$. Let $v$ denote the word obtained from $w$ after removing all letters $m$. Compute $\mstc_\beta(v)$ and denote it by $k'$. In Lemma \ref{lem:mstc =k} we show that $k'\leq \lambda_m$, therefore we can split $w$ into two subwords $w=(w^{(1)},w^{(2)})$ as follows
\begin{align}
w^{(1)}=(w_1\ldots w_{\alpha_m+k'}),\qquad
 w^{(2)} = (w_{\alpha_m+k'+1}\ldots w_{\alpha_m+\lambda_m})
\end{align}
The subwords $w^{(1)}$ and $w^{(2)}$ contain subwords of $v$ and subwords of $(m^{\gamma_m})$. Let $l$ be an integer such that $w^{(1)}$ has $\gamma_m-l$ letters $m$ and $w^{(2)}$ has $l$ letters $m$. Set $k=k'+l$, then there exists a unique pair of permutations $\rho$ and $\sigma$ as in \eqref{eq:rho sigma} such that
\begin{align*}
w^{(1)}=\rho(v_1\ldots v_{\beta_m+k'+l}\,m^{\gamma_m-l}),
\quad 
w^{(2)}=\sigma(v_{\beta_m+k'+l+1}\ldots v_{\lambda_m+\beta_m}\,m^{l})
\end{align*}
Hence the word $w$ is of the form
\begin{align}\label{eq:w v l}
w= (\rho(v_1\ldots v_{\beta_m+k'+l}\,m^{\gamma_m-l}),\sigma(v_{\beta_m+k'+l+1}\ldots v_{\lambda_m+\beta_m}\,m^{l}))
\end{align}
According to \eqref{eq:mstc beta w} the words $w$ which are of this form have a fixed value $\mstc_\beta(w)$ 
\begin{align}
\mstc_\beta(w)=k=k'+l=\mstc_\beta(v)+l
\end{align}
Therefore a word $w$ with a fixed value $\mstc_\beta(w)=k$ reproduces a unique collection of indices $(v,\rho,\sigma)$ with the restriction $\mstc_\beta(v)=k'$. This implies a bijection between all summation labels in \eqref{eq:psi coeff Y} and all $w\in S_{\gamma^{(m)}}$ such that $\mstc_\beta(w)=k$. Thus we can write \eqref{eq:psi coeff Y} as
\begin{align}\label{eq:psi coeff Z}
\psi^{(m)}_{k}
=
\sum_{\substack{w\in S_{\gamma^{(m)}}\\ \mstc_\beta(w)=k}}
q^{\inv(w)}
\end{align}
Inserting this formula into $\Psi^{(m)}(x)$ gives \eqref{eq:Psi word sum}. By induction it also holds for $\Psi^{(1)}(x)=\Phi_{\alpha,\beta}(x;q)$ and proves \eqref{eq:Phi mstc inv proof}. 

In the remaining part of this section we state and prove two lemmas which validate the two key equations \eqref{eq:binom inv w} and \eqref{eq:mstc beta w}.
\begin{lem}\label{lem:w}
Let $v=v_1\ldots v_L$ be a word with $v_i>m$ for every $i$. Fix integers $a,b,c$ with $a\geq b\geq 0$ and $0\leq c\leq L$. Consider the word
\begin{align}\label{eq:two shuffles}
w=(\rho(v_1\ldots v_c\,m^{a-b}),\sigma(v_{c+1}\ldots v_L\,m^b))
\end{align}
where $\rho\in S_{c+a-b}/(S_c\times S_{a-b})$ and $\sigma \in S_{L-c+b}/(S_{L-c}\times S_{b})$, i.e. $\rho$ and $\sigma$ are two permutations preserving the relative order of $v_i$ in the corresponding parts of the word $w$. Then
\begin{align}\label{eq:inv formula}
\sum_{\rho,\sigma}
q^{\normalfont{\inv}(w)}=q^{\normalfont{\inv}(v)+c b}\qbinom{a-b+c,c}\qbinom{L+b-c,b}
\end{align}
\end{lem}
\begin{proof}
We start from the sum on the left hand side of \eqref{eq:inv formula}. Using \eqref{eq:two shuffles} we write $w$ explicitly
\begin{align*}
\sum_{\rho,\sigma}
q^{\normalfont{\inv}(w)}= 
\sum_{\rho,\sigma}
q^{\normalfont{\inv}((\rho(v_1\ldots v_c\,m^{a-b}),\sigma(v_{c+1}\ldots v_L\,m^b)))}
\end{align*}
Since $\rho$ and $\sigma$  do not change the order of $v_i$'s we can extract the contribution of the inversion number of $(v_1\ldots v_L)$ and then replace all $v_i$ by a letter $d$ such that $d>m$. This gives us
\begin{align*}
\sum_{\rho,\sigma}
q^{\inv(w)}= 
q^{\inv(v)}
\sum_{\rho,\sigma}
q^{\inv((\rho(d^c,m^{a-b}),\sigma(d^{L-c},m^b)))}
\end{align*}
Next we split the inversion number of $(\rho(d^c,m^{a-b}),\sigma(d^{L-c},m^b))$ into the sum of the inversion numbers of $\rho(d^c,m^{a-b})$ and $\sigma(d^{L-c},m^b)$ taking into account the inversions between $d^c$ and $m^b$
\begin{align*}
\inv((\rho(d^c,m^{a-b}),\sigma(d^{L-c},m^b))) = 
\inv(\rho(d^c,m^{a-b})) +\inv(\sigma(d^{L-c},m^b))+c b
\end{align*}
This gives us
\begin{align}\label{eq:q inv proof}
\sum_{\rho,\sigma}
q^{\inv(w)}= 
q^{\inv(v)+c b}
\sum_{\rho}
q^{\inv(\rho(d^c,m^{a-b}))}
\sum_{\sigma}
q^{\inv(\sigma(d^{L-c},m^b))}
\end{align}
The sums over $\rho$ and $\sigma$ can be computed using the formula
\begin{align}\label{eq:q binom inv}
\sum_{\pi\in S_{A+B}/(S_A\times S_B)}
q^{\inv(\pi(1^{A},0^B))} = \qbinom{A+B,A}
\end{align}
which is a standard combinatorial interpretation of the $q$-binomial coefficient. Using \eqref{eq:q binom inv} in \eqref{eq:q inv proof} leads to \eqref{eq:inv formula} and proves the Lemma.
\end{proof}
With the following choice of the integers $L,a,b$ and $c$
\begin{align}\label{eq:L}
L=\gamma_{m+1}+\cdots +\gamma_{n+1}=\lambda_m+\beta_m
\end{align}
\begin{align*}
a=\gamma_m,
\quad 
b=k-k',
\quad 
c=\beta_m+k
\end{align*} 
the formula \eqref{eq:inv formula} becomes \eqref{eq:binom inv w}. 

\begin{lem}\label{lem:mstc =k}
For any  $v\in S_{\gamma^{(m+1)}}$ we have 
\begin{align}
\label{eq:mstc beta v}
\normalfont{\mstc}_\beta(v)\leq \lambda_m
\end{align}
As above we set $\normalfont{\mstc}_\beta(v)=k'$ and let $\rho$ and $\sigma$ be chosen as in \eqref{eq:rho sigma}. For $w\in S_{\gamma^{(m)}}$ of the form
\begin{align}\label{eq:w form}
w= (\rho(v_1\ldots v_{\beta_m+k}\,m^{\gamma_m-k+k'}),\sigma(v_{\beta_m+k+1}\ldots v_{\lambda_m+\beta_m}\,m^{k-k'}))
\end{align}
we have
\begin{align}\label{eq:mstc = k}
\normalfont{\mstc}_\beta(w)=k
\end{align}
\end{lem}
\begin{proof}
From Definition \ref{def:mstc beta} of $\mstc_\beta$ we have
\begin{align}\label{eq:mstc beta v def}
\mstc_\beta(v)=\st\left(\code((\std(v))^{-1})-\beta^{\gamma^{(m+1)}}\right)
\end{align}
The word $v$ is of length $\lambda_m+\beta_m$, therefore
\begin{align}\label{eq:mstc beta v bound 1}
\mstc_\beta(v)\leq \lambda_m+\beta_m
\end{align}
Let us argue that the $\alpha_0$ copies of $n+1$ contained in $v$ do not contribute to $\mstc_\beta(v)$ and thus further decrease this bound. 
The standardisation of $v$ maps its entries to $\{1\ldots \lambda_m+\beta_m\}$. In particular the $\alpha_0$ copies of $n+1$ are mapped to $\{\lambda_m+\beta_m-\alpha_0+1\ldots \lambda_m+\beta_m\}$. Let us denote their positions by $r_1\ldots r_{\alpha_0}$, then 
\begin{align*}
\std(v) = (\ldots \underset{\text{position }r_1}{\underbrace{\lambda_m+\beta_m-\alpha_0+1}}\ldots \underset{\text{position }r_{\alpha_0}}{\underbrace{\lambda_m+\beta_m}} \ldots)
\end{align*}
The last $\alpha_0$ entries of $(\std(v))^{-1}$ are given by $r_1\ldots r_{\alpha_0}$
\begin{align*}
\std(v)^{-1} = (\ldots r_1 \ldots r_{\alpha_0})
\end{align*}
Since these positions are increasing from left to right $r_1 < \cdots < r_{\alpha_0}$ their corresponding $\code$ entries are all equal to $0$
\begin{align*}
\code( \std(v)^{-1}) = (\ldots \underset{\alpha_0 \text{ entries}}{\underbrace{0\ldots 0}})
\end{align*}
The last $\alpha_0$ entries of $\beta^{\gamma^{(m+1)}}$ are also all equal to $0$ which means that the last $\alpha_0$ entries of the argument in $\st$ in \eqref{eq:mstc beta v def} are also equal to $0$ and thus do not contribute to the value of $\st$. This gives us
\begin{align}\label{eq:mstc beta v bound 2}
\mstc_\beta(v)\leq \lambda_m+\beta_m-\alpha_0\leq \lambda_m
\end{align}
where the second inequality follows from the condition $\alpha_0\geq \max(\beta)$. This proves \eqref{eq:mstc beta v}.

One consequence of \eqref{eq:mstc beta v} is that we can always split any word $w=(w_1\ldots w_{\alpha_m+\lambda_m})\in S_{\gamma^{(m)}}$ into two parts
\begin{align*}
(w_1\ldots w_{\alpha_m+k'}),\quad (w_{\alpha_m+k'+1}\ldots w_{\alpha_m+\lambda_m})
\end{align*}
where $k'=\mstc_\beta(v)$ and $v$ is the subword of $w$ obtained after deleting all letters $m$. If we specify how many entries $m$ these two subwords have then we can compute $\mstc_\beta(w)$. By Definition \ref{def:mstc beta}
\begin{align}\label{eq:mstc beta w X}
\mstc_\beta(w)=\st\left(\code((\std(w))^{-1})-\beta^{\gamma^{(m)}}\right)
\end{align}
We will compute explicitly the right hand side of \eqref{eq:mstc beta w X} for $w$ of the form \eqref{eq:w form} and show that it is equal to $k$. 

The smallest letters in $w$ are the $\gamma_m$ letters $m$. Therefore applying standardisation to $w$ assigns labels $j\in \{1\ldots \gamma_m\}$ to the letters $m$ and assigns labels $j> \gamma_m$ to the letters $v_i$. These labels become the letters of $\std(w)$. The positions of the letters $\{1\ldots \gamma_m\}$ in $\std(w)$ give precisely the first $\gamma_m$ entries of the inverse permutation of $\std(w)$. Let $p_1<\cdots<p_{\gamma_m}$ denote the positions of the copies of $m$ in $w$, then
\begin{align*}
((\std(w))^{-1})_i = p_i,
\qquad i=1\ldots \gamma_m
\end{align*}
In other words  $((\std(w))^{-1})$ is of the form 
\begin{align}\label{eq:std inv positions}
(\std(w))^{-1} =(p_1\ldots p_{\gamma_m}, s_{1}\ldots s_{L})
\end{align}
where $s_i$ are positions of the letters of $v$ and $L$ is given by \eqref{eq:L}. 
Next we compute the first $\gamma_m$ entries of the Lehmer code of $(\std(w))^{-1}$ using the definition \eqref{eq:Lehmer code}
\begin{align}\label{eq:Lehmer code of w X}
(\code(\std(w)^{-1}))_i=|\{j:i<j\leq L+\gamma_m, ~ (\std(w)^{-1})_i>(\std(w)^{-1})_j\}|
\end{align}
For the entries $i=1\ldots \gamma_m$ we can replace $(\std(w)^{-1})_i$ with $p_i$ and since $p_1<\cdots<p_{\gamma_m}$ we can replace $(\std(w)^{-1})_j$ with $s_{j-\gamma_m}$ (note $\gamma_m<j\leq L+\gamma_m$)
\begin{align}\label{eq:Lehmer code of w Y}
(\code(\std(w)^{-1}))_i=|\{j: 1\leq j\leq L, ~ p_i>s_j\}|,
\qquad i=1\ldots \gamma_m
\end{align}
Fix an $i=1\ldots \gamma_m$. The word $\std(w)^{-1}$, which is of the form \eqref{eq:std inv positions}, contains a total of $(p_i-1)$ entries which are smaller than $p_i$. Out of these $(p_i-1)$ entries $i-1$ are located to the left of $p_i$ because of the ordering $p_1<\cdots<p_{i}$. Hence to the right of $p_i$ we must have $(p_i-1)-(i-1)$ smaller entries in $\std(w)^{-1}$. This gives us
\begin{align}\label{eq:Lehmer code of w Z}
(\code(\std(w)^{-1}))_i= p_i-i,
\qquad i=1\ldots \gamma_m
\end{align}
Next we use the fact that the letters $m\in w$ are split into two parts
\begin{align*}
w= (\rho(v_1\ldots v_{\beta_m+k}\,m^{\gamma_m-k+k'}),\sigma(v_{\beta_m+k+1}\ldots v_{\lambda_m+\beta_m}\,m^{k-k'}))
\end{align*}
Because of this the entries of the Lehmer code $p_i-i$ in \eqref{eq:Lehmer code of w Z} must satisfy 
\begin{align}\label{eq:p leq}
&p_i-i\leq k+\beta_m,&\quad&\text{for }1\leq i\leq\gamma_m-k+k'
\\
\label{eq:p geq}
&p_i-i\geq k+\beta_m,&\quad& \text{for }\gamma_m-k+k'<i\leq\gamma_m
\end{align}
For example if $\rho$ is the identity permutation and $\sigma$ is the permutation which places all $m^{k-k'}$ before the $v$'s then in both cases \eqref{eq:p leq} and \eqref{eq:p geq} we have $p_i-i=k+\beta_m$. 
After setting
\begin{align}\label{eq:z_i}
z_i= p_i-i-\beta_m,
\qquad i=1\ldots \gamma_m
\end{align}
we can write the argument of $\st$ in \eqref{eq:mstc beta w X} as
\begin{align}\label{eq:z t}
\code((\std(w))^{-1})-\beta^{\gamma^{(m)}}&=\left( z_1\ldots z_{\gamma_m}, t_1\ldots t_L \right)\\
&=( \underset{\text{all }z_i\leq k}{\underbrace{z_1\ldots z_{\gamma_m-k+k'}}},\underset{\text{all }z_i\geq k}{\underbrace{z_{\gamma_m-k+k'+1}\ldots z_{\gamma_m}}}, t_1\ldots t_L )
\nonumber
\end{align}
where $t=(t_1\ldots t_L)=\code((\std(v))^{-1})-\beta^{\gamma^{(m+1)}}$ records the contributions from the letters of $v$ and in the second line we indicated explicitly the restrictions on $z$'s due to \eqref{eq:p leq} and \eqref{eq:p geq}. By the conditions of the Lemma $\st(t)=\mstc_{\beta}(v)=k'$ hence the subword $t$ has entries $t_{i_1}\ldots t_{i_{k'}}$ with $i_1<\cdots<i_{k'}$ and $t_{i_j}>k'-j$. Prepending the last $k-k'$ new entries $z$ (which are all $\geq k$) gives
\begin{align}\label{eq:max seq}
(z_{\gamma_m-k+k'+1}\ldots z_{\gamma_m},t_{i_1}\ldots t_{i_{k'}})>(k-1,k-2\ldots0)
\end{align}
componentwise. Thus we have
\begin{align}\label{eq:st X}
\st \left(z_{\gamma_m-k+k'+1}\ldots z_{\gamma_m}, t_1\ldots t_L \right)=k
\end{align}
Since the first $\gamma_m-k+k'$ letters $z$ in \eqref{eq:z t} do not exceed $k$ the sequence in \eqref{eq:max seq} cannot be extended using these letters, thus
\begin{align}\label{eq:st Y}
\st \left(z_{1}\ldots z_{\gamma_m}, t_1\ldots t_L \right)=k
\end{align}
Due to \eqref{eq:z t} and \eqref{eq:mstc beta w X} this proves \eqref{eq:mstc = k}.
\end{proof}

\section*{Acknowledgments}
The author gratefully acknowledges financial support of the Australian Research Council.

\end{document}